\documentclass[numbers,compress]{vmsta}
\usepackage{dcolumn}

\newtheorem{thm}{Theorem}  % [section]
\newtheorem{remark}{Remark}  % [section]
\newtheorem{lemma}{Lemma}  % [section]
\newtheorem{cor}{Corollary}  % [section]

\theoremstyle{definition}
\newtheorem{defin}{Definition}   %[section]

\newcolumntype{d}[1]{D{.}{.}{#1}}

\begin{document}
\begin{frontmatter}

\title{Generalization of Doob  decomposition Theorem.}
\author[a]{\inits{N.}\fnm{Nicholas}\snm{Gonchar}\corref{cor1}}\email{mhonchar@i.ua}
\cortext[cor1]{Corresponding author.}

\address[a]{Bogolyubov Institute for Theoretical Phisics of NAS, Kyiv, Ukraine
}

\markboth{N. Gonchar}{A sample document}

\begin{abstract}
In the paper, we introduce the notion of a local  regular supermartingale relative to a convex set of equivalent measures and prove for it an  optional  Doob  decomposition in the discrete case. This Theorem is a generalization of the  famous  Doob  decomposition onto the case of supermartingales relative to a convex set of equivalent measures. 
\end{abstract}


\end{frontmatter}

\section{Introduction.}

In the paper, we generalize Doob  decomposition
  for supermartingales relative to one measure  onto the case of supermartingales relative to a convex set of equivalent measures.  For supermartingales relative to one measure for  continuous time Doob's result was generalized in papers \citep{Meyer1, Meyer2}.

At the beginning, we prove the auxiliary statements giving sufficient conditions 
of the existence of maximal element in a maximal chain, of the existence of nonzero non-decreasing process such that the sum of a supermartingale and this process is again a supermartingale relative to a convex set of equivalent measures needed for the main Theorems. In  Theorem \ref{ctt5} we give sufficient conditions of the existence of the optional  Doob  decomposition for the special case as the set  of measures  is generated  by finite set of equivalent  measures  with bounded as below and above the Radon - Nicodym derivatives. After that, we introduce the notion of a regular supermartingale. Theorem \ref{ct4} describes   regular supermartingales. In  Theorem \ref{reww1} we give the necessary and sufficient conditions of regularity of supermartingales.
  Theorem \ref{kj1} describes  the structure of non-decreasing process for a  regular supermartingale. 
 Then we introduce the notion of a local regular supermartingale relative to a convex set of equivalent measures.
At last, we prove  Theorem \ref{hf1} asserting that if the optional decomposition for a supermartingale is valid, then it is local regular one. 
Essentially, Theorem \ref{hf1} and \ref{mars1} give the necessary and sufficient conditions of local regularity of supermartingale. 

After that, we prove auxiliary statements nedeed for the description of local regular supermartingales. Theorem \ref{mars12} gives the necessary and sufficient conditions for a  special  class of nonnegative supermartingales to be  local regular ones. In Theorems \ref{mars5} and \ref{mmars1} we describe a wide class of local regular supermartingales. On the basis of these Theorems we introduce  a certain class of local regular supermartingales  and prove Theorem \ref{mmars9} giving the necessary and sufficient conditions  for nonnegative uniformly integrable supermartingale to belong to this class. Using the results obtained we give examples of construction of  local regular supermartingales.
 At last, we prove also Theorem \ref{ttt9} giving  possibility to construct local regular supermartingales.

The optional decomposition for supermartingales plays fundamental role for risk assessment on incomplete markets  \cite{Gonchar2},     
\cite{Gonchar514}, \cite{Gonchar555},  \cite{Gonchar557}. 
Considered in the paper problem is generalization of corresponding one  that  appeared in mathematical finance  about optional decomposition for supermartingale and which is related with construction of superhedge strategy on incomplete financial markets. First, the optional decomposition for supermartingales was opened by  El Karoui N. and  Quenez M. C. \cite{KarouiQuenez} for diffusion processes. After that, Kramkov D. O. \cite{Kramkov}, \cite{FolmerKramkov1} proved the optional decomposition for nonnegative bounded supermartingales.  Folmer H. and Kabanov Yu. M.  \cite{FolmerKabanov1},  \cite{FolmerKabanov}  proved analogous result for an arbitrary supermartingale. Recently, Bouchard B. and Nutz M. \cite{Bouchard1} considered a class of discrete models and proved the necessary and sufficient conditions for validity of optional decomposition. 
Our statement of the problem unlike the above-mentioned one and it  is more general:  a supermartingale relative to a convex set of equivalent  measures is given  and it is necessary to find conditions on the supermartingale and the set of measures  under that  optional decomposition exists.
Generality of our statement of the problem is that we do not require that the considered  set of measures was generated by random process that is a local martingale  as it is done in the papers \cite{Bouchard1, KarouiQuenez, Kramkov, FolmerKabanov} and that is important for the proof of the  optional decomposition in these papers.
  \vskip 5mm

\section{Discrete case.}

We assume that on a measurable space $\{\Omega,\mathcal{F}\}$ a filtration    ${\mathcal{F}_{m}\subset\mathcal{F}_{m+1}}\subset\mathcal{F}, \ m=\overline{0, \infty},$ and a family of measures $ M$ on $\mathcal{F}$ are given. Further, we assume that ${\cal F}_0=\{\emptyset, \Omega \}.$
A random process $\psi={\{\psi_{m}\}_{m=0}^{\infty}}$ is said to be adapted one relative to the filtration $\{{\cal F}_m\}_{m=0}^{\infty}$ if $\psi_{m}$ is ${\cal F}_m$ measurable random value for all $m=\overline{0,\infty}.$
\begin{defin}
An adapted random process  $f={\{f_{m}\}_{m=0}^{\infty}}$ is said to be   a supermartingale relative to the filtration ${\cal F}_m,\ m=\overline{0,\infty},$ and the  family of measures   $ M$  if $E^P|f_m|<\infty, \ m=\overline{1, \infty}, \ P \in M,$ and the inequalities 
\begin{eqnarray}\label{pk11} 
E^P\{f_m|{\cal F}_k\} \leq f_k, \quad 0 \leq k \leq m, \quad m=\overline{1, \infty}, \quad P \in M,
\end{eqnarray}
are valid.
\end{defin}
We consider that the filtration  ${\cal F}_m,\ m=\overline{0,\infty},$ is fixed. Further,   for a supermartingale  $f$ we use as denotation $\{f_{m}, {\cal F}_m\}_{m=0}^{\infty} $ and denotation  $\{f_{m}\}_{m=0}^{\infty}.$

Bellow in a few theorems, we consider a  convex set of equivalent measures $M$ satisfying conditions: Radon -- Nicodym derivative of any measure $Q_1 \in M$ with respect to any measure  $Q_2 \in M$ satisfies inequalities 
\begin{eqnarray}\label{gon1}
0< {l\leq\frac{dQ_{1}}{dQ_{2}}\leq L}< \infty,\quad  Q_1, \   Q_2 \in M,  
\end{eqnarray}
where real numbers    $l,\ L$ do not depend on $Q_1, \ Q_2 \in M.$

 \begin{thm}\label{t1} Let ${\{f_{m}, {\cal F}_m\}_{m=0}^{\infty} }$ be a supermartingale concerning a convex set of equivalent measures $M$ satisfying conditions (\ref{gon1}).  If  for a certain measure  $P_{1}\in M$ there exist a natural number $1 \leq m_0<\infty,$  and ${\cal F}_{m_0-1}$ measurable nonnegative random value $  \varphi_{m_0},$  $P_1(\varphi_{m_0}>0)>0,$  such that   the inequality
$$f_{m_0-1}-E^{P_1}\{f_{m_0}|\mathcal{F}_{m_0-1}\}\geq \varphi_{m_0},$$
is valid, then  
\begin{eqnarray*}
f_{m_0-1}-E^{Q}\{f_{m_0}|\mathcal{F}_{m_0-1}\}\geq \frac{l }{1+L}\varphi_{m_0}, \quad   Q\in M_{\bar \varepsilon_0},
\end{eqnarray*}
where
$$  M_{\bar \varepsilon_0}=\{Q\in M, \ Q=(1-\alpha)P_{1}+\alpha P_{2},\ 0\leq\alpha \leq  \bar \varepsilon_0, \ P_{2}\in M \},  \quad P_{1}\in M,$$
 $$\bar \varepsilon_0=\frac{L}{1+L}. $$
\end{thm}

\begin{proof} Let   $B\in\mathcal{F}_{m_0-1}$ and  $Q=(1-\alpha)P_{1}+\alpha P_{2}, \ P_2 \in M,  \ 0<\alpha<1.$
Then 
\begin{eqnarray*}
\int\limits_{B}[f_{m_0-1}-E^{Q}\{f_{m_0}|\mathcal{F}_{m_0-1}\}]dQ=
\end{eqnarray*}
\begin{eqnarray*}
\int\limits_{B}E^{Q}\{[f_{m_0-1}-f_{m_0}]|\mathcal{F}_{m_0-1}\}dQ=
\end{eqnarray*}
\begin{eqnarray*}
\int\limits_{B}[f_{m_0-1}-f_{m_0}]dQ=
\end{eqnarray*}
\begin{eqnarray*}
(1-\alpha)\int\limits_{B}[f_{m_0-1}-f_{m_0}]dP_1+
\end{eqnarray*}
\begin{eqnarray*}
\alpha\int\limits_{B}[f_{m_0-1}-f_{m_0}]dP_2=
\end{eqnarray*}
\begin{eqnarray*} (1- \alpha)\int\limits_{B}[f_{m_0-1}-E^{P_1}\{f_{m_0}|\mathcal{F}_{m_0-1}\}]dP_1+
\end{eqnarray*}
\begin{eqnarray*}
\alpha\int\limits_{B}[f_{m_0-1}-E^{P_2}\{f_{m_0}|\mathcal{F}_{m_0-1}\}]dP_2\geq
\end{eqnarray*}
\begin{eqnarray*}
 (1- \alpha)\int\limits_{B}[f_{m_0-1}-E^{P_1}\{f_{m_0}|\mathcal{F}_{m_0-1}\}]dP_1=
\end{eqnarray*}
\begin{eqnarray*} (1- \alpha)\int\limits_{B}[f_{m_0-1}-E^{P_1}\{f_{m_0}|\mathcal{F}_{m_0-1}\}]\frac{dP_1}{dQ} dQ \geq
\end{eqnarray*}
$$  (1- \alpha) l \int\limits_{B} \varphi_{m_0}dQ\geq (1- \bar \varepsilon_0) l \int\limits_{B} \varphi_{m_0}dQ=\frac{l}{1+L}\int\limits_{B} \varphi_{m_0}dQ.$$
Arbitrariness of $B \in {\cal F}_{m_0-1}$ proves the needed inequality.
\end{proof}

\begin{lemma}\label{l1} Any supermartingale  ${\{f_m, {\cal F}_m\}_{m=0}^{\infty}}$ relative to  a family of measures  $ M$ for which there hold equalities   $E^{P}f_{m}=f_{0}, \ m=\overline{1,\infty},$ \  ${ P\in M},$ is a martingale  with respect to this family of measures and the filtration   ${\cal F}_m,\ m=\overline{1,\infty}.$
\end{lemma}
\begin{proof} The proof of  Lemma \ref{l1} see \cite{Kallianpur}.
\end{proof}
\begin{remark}\label{rem1} 
If the conditions of   Lemma \ref{l1} are valid, then there hold equalities
\begin{eqnarray}\label{jps1} 
E^P\{f_m|{\cal F}_k\}=f_k, \quad 0 \leq k \leq m, \quad m=\overline{1, \infty}, \quad P \in M.
\end{eqnarray}
\end{remark}

Let $f={\{f_{m}, {\cal F}_m \}_{m=0}^{\infty}}$ be a supermartingale relative to a convex set of equivalent measures $M$ and the filtration  ${\cal F}_m,\ m=\overline{0,\infty}.$ 
And let $G$ be a set of adapted non-decreasing  processes   $g={\{g_{m}\}_{m=0}^{\infty}}$, ${g_{0}=0},$ such that  $f+g={\{f_{m}+g_{m}\}_{m=0}^{\infty}}$ is a supermartingale concerning the family of measures    $M$ and the filtration   ${\cal F}_m,\ m=\overline{0,\infty}.$  

Introduce a partial ordering  $\preceq $ in the set of adapted non-decreasing processes $G.$

\begin{defin} We say that an adapted non-decreasing process ${g_{1}=\{g_{m}^{1}\}_{m=0}^{\infty}},$ ${g_{0}^{1}=0,} \ g_1 \in G,$  does not exeed an adapted non-decreasing process ${g_{2}=\{g_{m}^{2}\}_{m=0}^{\infty}},$ ${g_{0}^{2}=0}, \ g_{2} \in G,$  if   $P(g_{m}^{2}-g_{m}^{1}\geq 0)=1, \  m=\overline{1,\infty}.$ This partial ordering we denote by $g_{1}\preceq g_{2}.$
\end{defin}
For  every nonnegative adapted non-decreasing process $g=\{g_{m}\}_{m=0}^{\infty} \in G$ there exists limit  $\lim\limits_{m\to \infty}g_m$ which we denote by $g_{\infty}.$ 

\begin{lemma}\label{l2} Let  ${\tilde{G}}$ be a maximal chain in $G$ and for a certain  ${Q\in M}$ 
$\sup\limits_{g\in \tilde{G}}E_1^{Q}g=$ $\alpha^{Q}<\infty.$
Then there exists a sequence  $g^{s}=\{g_{m}^{s}\}_{m=0}^{\infty}\in \tilde G,$  ${s=1,2,...},$ such that
\begin{eqnarray*}
\sup\limits_{g\in \tilde{G}}E_1^{Q}g=\sup\limits_{s\geq 1}E_1^{Q}g^{s},
\end{eqnarray*}
where
$$ E_1^Qg=\sum\limits_{m=0}^\infty\frac{E^Q g_m}{2^m}, \quad g \in G.$$
\end{lemma}
\begin{proof}

Let  $0 < \varepsilon_s<\alpha^{Q}, \ s=\overline{1, \infty},$ be a sequence of real numbers satisfying conditions  $  \varepsilon_s > \varepsilon_{s+1}, \ \varepsilon_s \to 0,$ as $ s \to \infty.$ Then there exists an element  $g^s \in {\tilde{G}}$ such that
$\alpha^{Q} - \varepsilon_s < E_1^{Q}g^s \leq \alpha^{Q}, \  s=\overline{1,\infty}.$ The  sequence $g^s \in {\tilde{G}}, \  s=\overline{1,\infty},$ satisfies Lemma \ref{l2} conditions.
\end{proof}

\begin{lemma}\label{l3} 
 If a  supermartingale  ${\{f_{m}, {\cal F}_m\}_{m=0}^{\infty}}$   relative to  a convex set of equivalent measures  $M$ is such that
\begin{eqnarray}\label{buti1}
 |f_{m}|\leq \varphi,  \quad m=\overline{0,\infty}, \quad E^Q\varphi < T < \infty, \quad   Q \in M, 
\end{eqnarray}
where a real number $T$ does not depend on  $Q \in M,$ then every maximal chain  ${\tilde{G}} \subseteq G$ contains  a maximal element.
\end{lemma}
\begin{proof}  Let  $g=\{g_m\}_{m=0}^\infty$ belong to  $G,$ then
\begin{eqnarray*}
 {E^{Q}(f_{m}+\varphi+g_{m})\leq f_{0}+T, \quad {m=\overline{1,\infty}}, \quad  Q\in M}.
\end{eqnarray*}
Then inequalities  ${f_{m}+\varphi \geq 0}, \ m=\overline{1,\infty},$ yield
 \begin{eqnarray*}
 {E^{Q}g_{m}\leq f_{0}+T}, \quad {m=\overline{1,\infty}},\quad {\{g_{m}\}_{m=0}^{\infty} \in G}.
 \end{eqnarray*}
Introduce for a certain $Q \in M$ an expectation  for  $g=\{g_m\}_{m=0}^\infty \in G$ 
\begin{eqnarray*}
E_1^Qg=\sum\limits_{m=0}^\infty\frac{E^Q g_m}{2^m}, \quad g \in G.
 \end{eqnarray*}
Let $\tilde G \subseteq G$ be a certain maximal chain.
Therefore, we have inequality
 \begin{eqnarray*}
 \sup\limits_{g\in \tilde{G}}E_1^{Q}g=\alpha_{0}^{Q}\leq f_{0}+T< \infty,
 \end{eqnarray*}
 where   $Q \in M$  and is fixed.
 Due to  Lemma \ref{l2},
\begin{eqnarray*}
{\sup\limits_{g\in \tilde{G}}E_1^{Q}g=\sup\limits_{s\geq 1}E_1^{Q}g^{s}}.
\end{eqnarray*}
In consequence of the linear ordering of elements of  ${\tilde{G}},$
\begin{eqnarray*}
{\max\limits_{1\leq s\leq k}g^{s}=g^{s_{0}(k)}}, \quad {1\leq s_{0}(k)\leq k},
\end{eqnarray*}
where  $s_{0}(k)$  is one of elements of the set  $\{1,2, \ldots, k\}$ on which the considered maximum is reached, that is, $1 \leq s_{0}(k) \leq k,$
and, moreover, 
\begin{eqnarray*}
 {g^{s_{0}(k)} \preceq g^{s_{0}(k+1)}}.
 \end{eqnarray*}
It is evident that
\begin{eqnarray*}
\max\limits_{1\leq s \leq k}E_1^Q g^s = E_1^Q g^{s_0(k)}. 
 \end{eqnarray*}
So, we obtain 
\begin{eqnarray*}
{\sup\limits_{s\geq 1}E_1^{Q}g^{s}=\lim\limits_{k\rightarrow\infty} \max\limits_{1\leq s\leq k}E_1^{Q}g^{s}=\lim\limits_{k\rightarrow\infty}E_1^{Q}g^{s_{0}(k)}=E_1^{Q}\lim\limits_{k\rightarrow\infty}g^{s_{0}(k)}=E_1^{Q}g^{0}},
\end{eqnarray*}
where ${g^{0}=\lim\limits_{k\rightarrow\infty}g^{s_{0}(k)}}$, and that there exists, due to monotony of   ${g^{s_{0}(k)}}$.
Thus,
\begin{eqnarray*}
{\sup\limits_{s\geq 1} E_1^{Q}g^{s}=E_1^{Q}g^{0}=\alpha_{0}^{Q}}.
\end{eqnarray*}
 Show that  $g^0={\{g_{m}^{0}\}_{m=0}^{\infty}}$ is  a maximal element in  ${\tilde G}$.
It is evident that $g^0$ belongs to $G.$
For every element   ${g=\{g_{m}\}_{m=0}^{\infty}} \in \tilde G$ two cases are possible:\\
1) ${\exists k}$ such that  ${g\preceq g^{s_{0}(k)}}$.\\
2) ${\forall k \quad g^{s_{0}(k)}\prec g}$.\\
In the first case  ${g\preceq g^{0}}.$
In the second one from 2) we have ${g^{0}\preceq g}$.
At the same time 

\begin{eqnarray}\label{05}
E_1^{Q}g^{s_{0}(k)}\leq E_1^{Q}g.
\end{eqnarray} 
By passing to the limit in   (\ref{05}), we obtain 
\begin{eqnarray}\label{06}
E_1^{Q}g^{0}\leq E_1^{Q}g.
\end{eqnarray} 
The strict inequality in   (\ref{06}) is impossible, since  ${E_1^{Q}g^{0}=\sup\limits_{g\in\tilde{G}}E_1^{Q}g}$.
Therefore, 
\begin{eqnarray}\label{07}
E_1^{Q}g^{0}= E_1^{Q}g.
\end{eqnarray} 
 The inequality  ${g^{0}\preceq g}$  and the equality (\ref{07}) imply that  ${g=g^{0}.}$
\end{proof}

Let $M$ be a convex set of equivalent probability measures on   $\{\Omega, {\cal F}\}.$ Introduce into $M$ a metric  $|Q_1 -Q_2|=\sup\limits_{A \in {\cal F}}|Q_1(A) - Q_2(A)|,\  Q_1, \ Q_2 \in M.$ 

\begin{lemma}\label{qhon13}
Let $\{f_m, {\cal F}_m \}_{m=0}^\infty$  be a supermartingale relative to a compact convex set of equivalent measures  $M$ satisfying conditions (\ref{gon1}).
 If for every set of measures  $\{P_1, P_2, \ldots, P_s\}, \ s<\infty, \ P_i \in M, \ i=\overline{1,s},$
there exist a natural number   $1 \leq m_0<\infty,$  and depending on this
set of measures  ${\cal F}_{m_0-1}$ measurable nonnegative  random variable $\Delta_{m_0}^s,$   $P_1(\Delta_{m_0}^s>0)>0,$ satisfying conditions
\begin{eqnarray}\label{qhon14}
f_{m_0-1}-E^{ P_i}\{f_{m_0}|{\cal F}_{m_0-1}\} \geq \Delta_{m_0}^s, \quad i=\overline{1,s},  
\end{eqnarray}
then the set  $G$ of  adapted non-decreasing processes   $g=\{g_m\}_{m=0}^\infty , \  g_0=0,$ for which  $\{f_m+g_m\}_{m=0}^\infty$ is a supermartingale relative to the set of measures   $M$ contains nonzero element.
\end{lemma}
\begin{proof} 
For any point  $P_0 \in  M$ let us define a set of measures
\begin{eqnarray}\label{qd6}
M^{P_0, \bar \varepsilon_0}=\{Q \in  M, \ Q=(1-\alpha)P_0+\alpha P, \  P \in  M , \ 0 \leq \alpha \leq  \bar\varepsilon_0\},
\end{eqnarray} 
 $$\bar\varepsilon_0=\frac{L}{1+L}.$$
Prove that the set of measures  $M^{P_0, \bar \varepsilon_0}$ contains some ball of a positive  radius, that is, there exists a real number $\rho_0>0$ such that  
$M^{P_0, \bar \varepsilon_0}\supseteq C(P_0,\rho_0),$ where $C(P_0,\rho_0)=\{P\in  M,\ |P_0 - P |< \rho_0\}.$ 

Let  $C(P_0,\tilde \rho)=\{P\in  M, \ |P_0 - P |< \tilde \rho\}$ be an open ball in  $ M$ with the center at the point $P_0$ of a radius $0<\tilde \rho<1.$ Consider a map of the set $  M$ into itself given by the law: $f(P)=(1-\bar \varepsilon_0)P_0+\bar \varepsilon_0 P, \ P \in  M.$

The mapping  $ f(P)$   maps  an open ball  $C(P_2^{'},\delta)=\{P\in  M, |P_2^{'} - P |< \delta \}$  with the center at the point  $P_2^{'} $ of a radius  $\delta>0$  into an open ball with the center at the point  $(1- \bar \varepsilon_0)P_0+ \bar \varepsilon_0 P_2^{'}$ of the radius  $\bar \varepsilon_0\delta,$
since  $|(1-\bar \varepsilon_0)P_0+\bar \varepsilon_0 P_2^{'} - (1-\bar \varepsilon_0)P_0-\bar \varepsilon_0 P|=\bar \varepsilon_0|P_2^{'} - P|<\bar \varepsilon_0\delta.$ Therefore, an image of an open set  $ M_0 \subseteq  M$ is an open set $f(M_0)\subseteq  M,$ thus $ f(P)$ is an open mapping. Since $f(P_0)=P_0,$ then the image of the ball
 $C(P_0,\tilde \rho)=\{P\in  M,\ |P_0 - P |< \tilde\rho\}$ is a ball $C(P_0,\bar \varepsilon_0\tilde \rho)=\{P\in M,\ |P_0 - P |< \bar \varepsilon_0\tilde\rho\}$ and it is contained in  $f( M).$ 
 Thus, inclusions  $M^{P_0, \bar \varepsilon_0} \supseteq f( M) \supseteq C(P_0,\bar \varepsilon_0\tilde\rho)$ are valid.  Let us put $\bar \varepsilon_0\tilde\rho=\rho_0.$ Then we have $M^{P_0, \bar \varepsilon_0}\supseteq C(P_0,\rho_0),$ where  $C(P_0,\rho_0)=\{P\in  M,\ |P_0 - P |< \rho_0\}.$ 
Consider an open covering  $\bigcup\limits_{P_0 \in M}C(P_0, \rho_0)$ of the compact set  $ M.$  Due to compactness of  $ M,$ there exists a finite subcovering 
\begin{eqnarray}\label{qalmyd7}
 M=\bigcup\limits_{i=1}^v C(P_0^i, \rho_0) 
\end{eqnarray} 
with the center at  the points  $P_0^i \in M,  \ i=\overline{1, v}, $ and a covering by sets $ M^{P_0^i, \bar \varepsilon_0} \supseteq C(P_0^i, \rho_0),  \ i=\overline{1, v}, $
\begin{eqnarray}\label{qd7}
 M=\bigcup\limits_{i=1}^v M^{P_0^i, \bar \varepsilon_0}. 
\end{eqnarray} 
 
Consider the set of measures $P_0^i \in M,  \ i=\overline{1, v}. $
From  Lemma  \ref{qhon13} conditions, there exist a natural number   $1 \leq m_0<\infty,$  and depending on the
set of measures  $P_0^i \in M,  \ i=\overline{1, v}, $  ${\cal F}_{m_0-1}$ measurable nonnegative  random variable $\Delta_{m_0}^v,$  $P_0^1(\Delta_{m_0}^v>0)>0,$  such that
\begin{eqnarray}\label{cck1}
f_{m_0-1}-E^{ P_0^i}\{f_{m_0}|{\cal F}_{m_0-1}\} \geq \Delta_{m_0}^v, \quad i=\overline{1,v}.
\end{eqnarray}
 Due to  Theorem \ref{t1}, we have 
\begin{eqnarray}\label{ccj1}
f_{m_0-1} - E^Q\{f_{m_0}|{\cal F}_{m_0-1}\} \geq\frac{l}{1+L} \Delta_{m_0}^v=\varphi_{m_0}^v,   \quad Q \in M.
\end{eqnarray}
The last inequality imply 
\begin{eqnarray}\label{qk4}
E^Q\{f_{m_0-1}|{\cal F}_{s}\} - E^Q\{f_{m_0}|{\cal F}_{s}\} \geq E^{Q}\{ \varphi_{m_0}^v|{\cal F}_{s}\}, \quad Q \in M, \quad s <m_0.
\end{eqnarray}
But $ E^Q\{f_{m_0-1}|{\cal F}_{s}\} \leq f_s, \ s < m_0.$  Therefore,
\begin{eqnarray}\label{qk5}
 f_s - E^Q\{f_{m_0}|{\cal F}_{s}\} \geq E^{Q}\{\varphi_{m_0}^v|{\cal F}_{s}\}, \quad Q \in M, \quad s <m_0.
\end{eqnarray}
Since 
\begin{eqnarray}\label{qk6}
 f_{m_0} - E^Q\{f_{m}|{\cal F}_{m_0}\} \geq 0,  \quad Q \in M,  \quad m \geq m_0,
\end{eqnarray}
we have
\begin{eqnarray}\label{qk7}
E^Q\{ f_{m_0}|{\cal F}_{s}\} - E^Q\{f_{m}|{\cal F}_{s}\} \geq 0,  \quad Q \in M, \quad s <m_0, \quad m \geq m_0.
\end{eqnarray}
Adding  (\ref{qk7}) to (\ref{qk5}), we obtain
\begin{eqnarray}\label{qk8}
  f_s  - E^Q\{f_{m}|{\cal F}_{s}\}\geq E^{Q}\{ \varphi_{m_0}^v|{\cal F}_{s}\},  \quad Q \in M, \quad s <m_0, \quad m \geq m_0,
\end{eqnarray}
or
\begin{eqnarray}\label{qk9}
  f_s  - E^Q\{f_{m}|{\cal F}_{s}\}\geq E^{Q}\{\varphi_{m_0}^v|{\cal F}_{s}\}\chi_{[m_0,\infty)}(m) -  \varphi_{m_0}^v\chi_{[m_0,\infty)}(s), 
\end{eqnarray}
$$   Q \in M, \quad s \leq m_0, \quad m \geq m_0.$$
Introduce an adapted non-decreasing process
\begin{eqnarray*}
g^{m_0}=\{g_m^{m_0}\}_{m=0}^\infty, \quad g_m^{m_0}= \varphi_{m_0}^v\chi_{[m_0,\infty)}(m),
\end{eqnarray*}
where $\chi_{[m_0,\infty)}(m)$ is an indicator function of the set $ [m_0,\infty). $
Then   (\ref{qk9}) implies  that
\begin{eqnarray*}
E^Q\{ f_m+ g_m^{m_0}|{\cal F}_k\}\leq f_k +g_k^{m_0}, \quad  0 \leq k \leq m, \quad Q \in M.
\end{eqnarray*}
\end{proof}
 In the Theorem \ref{ctt5} a convex set of equivalent  measures 
\begin{eqnarray}\label{self100}
M =\{Q, \ Q=\sum\limits_{i=1}^n \alpha_iP_i, \  \alpha_i \geq 0, \ i=\overline{1, n}, \ \sum\limits_{i=1}^n \alpha_i=1\}
\end{eqnarray}
satisfies conditions
\begin{eqnarray}\label{self101}
0<l \leq \frac{dP_i}{dP_j} \leq L < \infty,\quad  i,j=\overline{1,n},
\end{eqnarray} 
where  $l,\ L $  are real  numbers.

 Denote by  $G$  the set of all  adapted non-decreasing processes  $g=\{g_m\}_{m=0}^{\infty},$ $ g_0=0,$ such that  $\{f_m+g_m\}_{m=0}^{\infty}$ is a supermartingale relative to all measures from $M.$
\begin{thm}\label{ctt5} 
Let   a supermartingale ${\{f_{m},  {\cal F}_m\}_{m=0}^{\infty}}$ relative to the set of measures (\ref{self100})   satisfy the conditions (\ref{buti1}),  and let  there exist a natural number $1 \leq m_0<\infty,$ and ${\cal F}_{m_0-1}$ measurable nonnegative  random value $\varphi_{m_0}^n,$  $P_1(\varphi_{m_0}^n>0)>0,$ such that
\begin{eqnarray}\label{cqkgon14}
f_{m_0-1}-E^{ P_i}\{f_{m_0}|{\cal F}_{m_0-1}\} \geq \varphi_{m_0}^n, \quad i=\overline{1,n}.
\end{eqnarray}
If  for  the maximal  element 
  $g^{0}=\{g_m^{0}\}_{m=0}^{\infty}$ in a certain  maximal chain  $\tilde G \subseteq G$ the  equalities  
\begin{eqnarray}\label{c5}
E^{P_i}(f_{\infty}+g_{\infty}^{0})=f_{0}, \quad P_i \in M,  \quad i=\overline{1,n},
\end{eqnarray} 
are valid, where $f_{\infty}=\lim\limits_{m \to \infty}f_m,$ $g_{\infty}^0=\lim\limits_{m \to \infty}g_m^0,$ then there hold equalities 
\begin{eqnarray}\label{c5c}
E^{P}\{f_{m}+g_{m}^{0}|{\cal F}_k\}=f_{k}+g_{k}^{0}, \quad 0 \leq k \leq m, \quad  m=\overline{1, \infty}, \quad P \in M. 
\end{eqnarray} 
\end{thm}
\begin{proof} 
The set $M$ is compact one  in the introduced metric topology. From the inequalities  (\ref{cqkgon14})  and the formula
\begin{eqnarray}\label{ns23}
 E^{Q}\{f_{m_0}|{\cal F}_{m_0-1}\}=\frac{\sum\limits_{i=1}^{n}\alpha_i E^{P_1}\{\varphi_i|{\cal F}_{m_0-1}\}E^{P_i}\{f_{m_0}|{\cal F}_{m_0-1}\}}{\sum\limits_{i=1}^{n}\alpha_i E^{P_1}\{\varphi_i|{\cal F}_{m_0-1}\}}, \quad Q \in M,
\end{eqnarray}
where $\varphi_i=\frac{dP_i}{dP_1},$ we obtain
\begin{eqnarray}\label{dargon14}
f_{m_0-1}-E^{Q}\{f_{m_0}|{\cal F}_{m_0-1}\} \geq \varphi_{m_0}^n, \quad Q \in M.
\end{eqnarray}
The inequalities (\ref{self101}) lead to inequalities
\begin{eqnarray}\label{mazgon14}
\frac{1}{nL} \leq \frac{dQ}{dP} \leq n L, \quad P, Q \in M.
\end{eqnarray}
  Inequalities (\ref{dargon14}) and (\ref{mazgon14}) imply that  conditions of  Lemma \ref{qhon13}  are satisfied  for any set of measures $Q_1,\ldots,Q_s \in M.$ Hence, it follows  that the set  $G$   contains nonzero element. Let $\tilde G \subseteq G$ be a maximal chain in  $ G$ satisfying condition of  Theorem \ref{ctt5}.
Denote by  $g^0=\{g_m^0\}_{m=0}^\infty,\ g_0^0=0,$ a maximal element  in $\tilde G \subseteq G.$ 
  Theorem \ref{ctt5} and Lemma  \ref{l3} yield  that as $\{f_m\}_{m=0}^{\infty}$ and $\{g_m^0\}_{m=0}^\infty$ are uniformly integrable relative to each measure from $M.$ There exist   therefore limits
$$ \lim\limits_{m \to \infty}f_m=f_{\infty}, \quad \lim\limits_{m \to \infty}g_m^0=g_{\infty}^0$$
with probability 1.
Due to  Theorem  \ref{ctt5}  condition,  in this  maximal chain 
$$E^{P_i}(f_{\infty}+g_{\infty}^0)=f_0, \quad P_i \in M, \quad  i=\overline{1,n} .  $$
Since $ \{f_m+g_m^0\}_{m=0}^\infty$ is a supermartingale concerning all measures from $M,$ we have
\begin{eqnarray}\label{cmykald8}
 E^{P_i}(f_{m}+g_m^0) \leq   E^{P_i}(f_{k}+g_k^0)\leq f_0, \quad k< m, \quad m=\overline{1,\infty}, \quad  P_i \in M.
\end{eqnarray}
By passing to the limit in (\ref{cmykald8}), as $m \to \infty,$ we obtain
\begin{eqnarray}\label{cmykald7} 
f_{0}=E^{P_i}(f_{\infty}+g_{\infty}^0) \leq  E^{P_i}(f_{k}+g_k^0)\leq f_0,  \quad k= \overline{1,\infty}, \quad  P_i \in M. 
\end{eqnarray}
So, $ E^{P_i}(f_{k}+g_k^0)= f_0, \ k=\overline{1,\infty}, \  P_i \in M, \ i=\overline{1,n}.$
Taking into account Remark \ref{rem1} we have
\begin{eqnarray}\label{cd8}
 E^{P_i}\{f_{m}+g_m^0|{\cal F}_k\}=f_k+ g_k^0, \quad 0 \leq  k \leq m, \ m=\overline{1,\infty}, \ P_i \in M, \ i=\overline{1,n}.
\end{eqnarray}
Hence,
$$E^{P}\{f_{m}+g_m^0|{\cal F}_{k}\}=$$
\begin{eqnarray}\label{ck1}
 \frac{\sum\limits_{i=1}^{n}\alpha_i E^{P_1}\{\varphi_i|{\cal F}_{k}\}E^{P_i}\{f_{m}+g_m^0|{\cal F}_{k}\}}{\sum\limits_{i=1}^{n}\alpha_i E^{P_1}\{\varphi_i|{\cal F}_{k}\}}=f_k+g_k^0,  \quad 0 \leq k \leq m, \quad P \in M,
\end{eqnarray}
where $ \varphi_i=\frac{dP_i}{dP_1}, \ i=\overline{1,n}.$ Theorem  \ref{ctt5} is proven.
 \end{proof}
Let $M$ be a convex set of equivalent measures. Bellow,  $G_s$ is  a set  of adapted non-decreasing  processes $\{g_m\}_{m=0}^\infty,$ $ \ g_0=0,$ for which $\{f_m+g_m\}_{m=0}^\infty$ is a supermartingale relative
to all measures from 
 \begin{eqnarray}\label{Wov3}
\hat M_s=\{Q, Q=\sum\limits_{i=1}^s\gamma_i \hat P_i, \ \gamma_i \geq 0, \ i=\overline{1,s}, \ \sum\limits_{i=1}^s\gamma_i=1\},
\end{eqnarray}
where  $\hat P_1, \ldots,\hat P_s \in M$ and satisfy conditions
 \begin{eqnarray}\label{fifa1}
0< {l\leq\frac{d\hat P_i}{d\hat P_j}\leq L}< \infty,\quad  i,j=\overline{1,s},  
\end{eqnarray}
 $l, L$ are  real numbers depending on the set of measures  $\hat P_1, \ldots,\hat P_s \in M.$

\begin{defin}\label{1000h} 
Let   a  supermartingale   $\{f_m, {\cal F}_m\}_{m=0}^\infty$ relative to  a convex set of equivalent measures  $M $ satisfy  conditions (\ref{buti1}). 
 We  call it  regular one if for every   set of measures (\ref{Wov3}) satisfying conditions (\ref{fifa1}) there exist a natural number $1 \leq m_0<\infty,$ and   ${\cal F}_{m_0-1}$ measurable nonnegative random value $\varphi_{m_0}^s,$ $\hat P_1(\varphi_{m_0}^s>0)>0, $ such that  the inequalities 
$$f_{m_0-1}- E^{\hat P_{i}}\{f_{m_0}|{\cal F}_{m_0-1}\}\geq \varphi_{m_0}^s, \quad i=\overline{1,s},$$
 hold and for the maximal element $g^s=\{g_m^s\}_{m=0}^\infty$
in a certain  maximal chain $\tilde G_s \subseteq G_s$  the equalities
\begin{eqnarray}\label{Wov1}
E^{\hat P_i}\{f_m +g_m^s|{\cal F}_{k}\}=f_{k}+g_{k}^s, \quad 0 \leq  k \leq m, \quad i=\overline{1,s},\quad m=\overline{1, \infty},
\end{eqnarray} 
are valid. Moreover, there exists  an adapted nonnegative process $\bar g^0=\{\bar g_m^0\}_{m=0}^\infty, $ $ \bar g_0^0=0,$  $ E^P \bar g_m^0<\infty, \  m=\overline{1, \infty}, \  P \in M,$ not depending on the set of measures 
$\hat P_1, \ldots,\hat P_s$  such that
\begin{eqnarray}\label{WOv2}
E^{\hat P_i}\{g_m^s- g_{m-1}^s|{\cal F}_{m-1}\}=E^{\hat P_i}\{\bar g_m^0|{\cal F}_{m-1}\}, \quad m=\overline{1,\infty}, \quad i=\overline{1,s}.
\end{eqnarray}
\end{defin}
The next Theorem describes  regular supermartingales.
\begin{thm}\label{ct4}   
  Let $\{f_m, {\cal F}_m\}_{m=0}^\infty$ be a regular   supermartingale  relative to  a convex set of equivalent measures  $M. $  
Then for the maximal element   $g^0=\{g_m^0\}_{m=0}^{\infty}$ in  a certain  maximal chain  $\tilde G \subseteq  G$ the equalities  
\begin{eqnarray*}%\label{125}
E^{P_0}(f_m +g_m^0)=f_0, \quad m=\overline{1,\infty},  \quad P_0 \in M,
\end{eqnarray*} 
are valid. There exists   a martingale   $\{\bar M_m,{\cal F}_m\}_{m=0}^\infty$  relative to the family of measures   $M$  such that 
\begin{eqnarray*}%\label{126} 
 f_m=\bar M_m- g_m^0, \quad m=\overline{1,\infty}. 
\end{eqnarray*}
Moreover, for the martingale  $\{\bar M_m,{\cal F}_m\}_{m=0}^\infty$   the representation 
\begin{eqnarray*}%\label{127}
\bar M_m=E^{P_0}\{f_{\infty}+ g_{\infty}|{\cal F}_m\},\quad m=\overline{1,\infty},  \quad P_0 \in M,
\end{eqnarray*}
holds, where  $ f_{\infty}+ g_{\infty}=\lim\limits_{m \to \infty}(f_m+g_m).$
\end{thm} 
\begin{proof}   For any finite set of measures $P_1, \ldots,P_n,$ $P_i \in M, \ i=\overline{1,n},$
let us introduce into consideration two sets of measures 
\begin{eqnarray*}
 M_n=\{P, \ P=\sum\limits_{i=1}^n\alpha_iP_i, \  \alpha_i \geq 0, \ i=\overline{1,n}, \ \sum\limits_{i=1}^n\alpha_i=1\},
\end{eqnarray*}
\begin{eqnarray*}
\tilde M_n=\{P, \  P=\sum\limits_{i=1}^n\alpha_iP_i, \  \alpha_i > 0, \  i=\overline{1,n}, \ \sum\limits_{i=1}^n\alpha_i=1\}.
\end{eqnarray*}
Let $\hat P_1, \ldots, \hat P_s$ be a certain subset of measures    from $\tilde M_n.$  For every measure $\hat P_i \in \tilde M_n$ the representation   
 $\hat P_i=\sum\limits_{k=1}^n\alpha_k^iP_k $ is valid, where $\alpha_k^i>0,  \ i=\overline{1,s}, \ k=\overline{1,n}.$   The representation for $\hat P_i, \ i=\overline{1,s}, $ imply the validity  of inequalities
\begin{eqnarray*}
0<l=\min\limits_{i,j}\frac{\min\limits_{k}\alpha_k^i}{\max\limits_{k}\alpha_k^j} \leq \frac{d\hat P_i}{d\hat P_j}\leq \max\limits_{i,j}\frac{\max\limits_{k}\alpha_k^i}{\min\limits_{k}\alpha_k^j}=L < \infty, \quad i,j=\overline{1,s}.
\end{eqnarray*}
 Denote by $G_s$  a set  of adapted non-decreasing  processes $\{g_m\}_{m=0}^\infty, \ g_0=0,$ for which $\{f_m+g_m\}_{m=0}^\infty$ is a supermartingale relative
to all measures from 
  $$\hat M_s=\{Q, \  Q=\sum\limits_{i=1}^s\gamma_i \hat P_i, \ \gamma_i \geq 0, \ i=\overline{1,s}, \ \sum\limits_{i=1}^s\gamma_i=1\}.$$
In accordance with the definion of a regular supermartingale,  there exist a natural number $1 \leq  m_0<\infty,$ and   ${\cal F}_{m_0-1}$ measurable nonnegative random value $\varphi_{m_0}^s,$ $\hat P_1(\varphi_{m_0}^s>0)>0, $  such that the inequalities there hold  
$$f_{m_0-1}- E^{\hat P_{i}}\{f_{m_0}|{\cal F}_{m_0-1}\}\geq \varphi_{m_0}^s, \quad i=\overline{1,s},$$
and for a maximal element $g^s=\{g_m^s\}_{m=0}^\infty$
in a certain maximal chain $\tilde G_s \subseteq G_s$ there hold equalities (\ref{Wov1}), (\ref{WOv2}). Equalities (\ref{WOv2}) yield the equalities
$$E^{Q}\{g_{m}^s -  g_{m-1}^s|{\cal F}_{m-1}\}=$$
\begin{eqnarray*}\label{ck1}
 \frac{\sum\limits_{i=1}^{s}\gamma_i E^{\hat P_1}\{\hat \varphi_i|{\cal F}_{m-1}\}E^{\hat P_i}\{g_{m}^s - g_{m-1}^s|{\cal F}_{m-1}\}}{\sum\limits_{i=1}^{s}\gamma_i E^{\hat P_1}\{\hat \varphi_i|{\cal F}_{m-1}\}}=   
\end{eqnarray*}
\begin{eqnarray}\label{Wau1}
\frac{\sum\limits_{i=1}^{s}\gamma_i E^{\hat P_1}\{\hat \varphi_i|{\cal F}_{m-1}\}E^{\hat P_i}\{\bar g_{m}^0|{\cal F}_{m-1}\}}{\sum\limits_{i=1}^{s}\gamma_i E^{\hat P_1}\{\hat \varphi_i|{\cal F}_{m-1}\}}=E^{Q}\{\bar g_{m}^0|{\cal F}_{m-1}\}, 
\end{eqnarray}
$$  m=\overline{1, \infty}, \quad Q \in \hat M_s.$$
where $\hat \varphi_i=\frac{d\hat P_i}{d\hat P_1}, \ i=\overline{1,n}.$
Taking into account the equalities (\ref{Wov1}), we obtain
$$E^{Q}\{f_m+g_{m}^s |{\cal F}_{m-1}\}=$$
\begin{eqnarray*}\label{Fck1}
 \frac{\sum\limits_{i=1}^{s}\gamma_i E^{\hat P_1}\{\hat \varphi_i|{\cal F}_{m-1}\}E^{\hat P_i}\{f_m+g_{m}^s|{\cal F}_{m-1}\}}{\sum\limits_{i=1}^{s}\gamma_i E^{\hat P_1}\{\hat \varphi_i|{\cal F}_{m-1}\}}=   
\end{eqnarray*}
\begin{eqnarray}\label{FWau1}
f_{m-1}+g_{m-1}^s, \quad  m=\overline{1, \infty}, \quad Q \in \hat M_s.
\end{eqnarray}

Thus, we have
\begin{eqnarray}\label{Wau2}
E^{Q}\{g_{m}^s -  g_{m-1}^s|{\cal F}_{m-1}\}=E^{Q}\{\bar g_{m}^0|{\cal F}_{m-1}\},   \quad m=\overline{1, \infty},   \quad Q \in \hat M_s.
\end{eqnarray}
\begin{eqnarray}\label{FWau2}
E^{Q}\{f_m+g_{m}^s|{\cal F}_{m-1}\}=f_{m-1}+g_{m-1}^s,   \quad m=\overline{1, \infty},   \quad Q \in \hat M_s.
\end{eqnarray}
Let us introduce into consideration a random process $\{N_m, {\cal F}_m\}_{m=0}^\infty,$
 where
$$ N_0=f_0, \quad N_m=f_m+\sum\limits_{i=1}^m \bar g_m^0, \quad m=\overline{1, \infty}.$$
It is evident that $E^{Q}|N_m|< \infty, \ m=\overline{1, \infty}, \ Q \in  \hat M_s.$ The definition of  $\{N_m, {\cal F}_m\}_{m=0}^\infty$ and the formulae (\ref{Wau2}), ( \ref{FWau2}) yield
$$E^{Q}\{N_{m-1}- N_m |{\cal F}_{m-1}\}=
E^{Q}\{f_{m-1} - f_m - \bar g_m^0|{\cal F}_{m-1}\}=$$ $$=E^{Q}\{g_m^s- g_{m-1}^s -\bar g_m^0|{\cal F}_{m-1}\}=0, \quad m=\overline{1, \infty}, \quad \ Q \in  \hat M_s.$$
The last equalities imply 
$$ E^{Q}\{N_m|{\cal F}_{m-1}\}=N_{m-1}, \quad m=\overline{1, \infty},  \quad Q \in \hat M_s.$$
Due to arbitrariness of the set of measures $\hat P_1, \ldots, \hat P_s, $ $ \hat P_i \in \tilde M_n,$ we have
\begin{eqnarray}\label{wy1}
E^{P}\{N_m|{\cal F}_{m-1}\}=N_{m-1}, \quad  P \in \tilde M_n,\quad m=\overline{1, \infty}.
\end{eqnarray}
So,  the set $G_0$   of  adapted non-decreasing processes $\{g_m\}_{m=0}^\infty,$ $ \ g_0=0,$ for which $\{f_m+g_m\}_{m=0}^\infty$ is a supermartingale relative
to all measures from $\tilde M_n$ contains nonzero element $\tilde g^0=\{\tilde g_m^0\}_{m=0}^\infty, \ \tilde g_0^0=0,$ $\tilde g_m^0=\sum\limits_{i=1}^m \bar g_m^0, \ m=\overline{1, \infty},$ which is  a maximal element in a  maximal chain $\tilde G_0$ containing this element. Really, if  $g^0=\{ g_m^0\}_{m=0}^\infty, \  g_0^0=0,$ is a maximal element 
 in the maximal chain $\tilde G_0 \subseteq G_0,$ then there hold inequalities
\begin{eqnarray}\label{my1}
E^{P_0}\{f_m+ g_m^0|{\cal F}_{k}\} \leq f_{k}+ g_{k}^0,\quad m=\overline{1,\infty},  \quad 1 \leq k \leq m,\quad P_0 \in \tilde M_n,
\end{eqnarray}
\begin{eqnarray}\label{my2}
E^{P_0}(f_m+ g_m^{0})\leq f_0 ,\quad m=\overline{1,\infty},  \quad P_0 \in \tilde M_n.
\end{eqnarray}
and inequality $\tilde g^0 \preceq g^0$  meaning that $\tilde g_m^{0} \leq  g_m^{0}, \ m=\overline{0, \infty}.$
Equalities (\ref{wy1}) yield
\begin{eqnarray}\label{ty2}
E^{P_0}(f_m+\tilde g_m^{0})= f_0 ,\quad m=\overline{1,\infty},  \quad P_0 \in \tilde M_n.
\end{eqnarray}
 Inequalities  (\ref{my2})  and equalities  (\ref{ty2}) imply
\begin{eqnarray}\label{ty3}
f_0 \geq E^{P_0}(f_m+ g_m^{0})\geq  E^{P_0}(f_m+\tilde g_m^{0})= f_0 ,\quad m=\overline{1,\infty},  \quad P_0 \in \tilde M_n.
\end{eqnarray}
The last inequalities lead to equalities
\begin{eqnarray}\label{ty4}
 E^{P_0}( g_m^{0}- \tilde g_m^{0})= 0, \quad m=\overline{1,\infty},  \quad P_0 \in \tilde M_n.
\end{eqnarray}
But
\begin{eqnarray}\label{ty5}
 g_m^{0}- \tilde g_m^{0}\geq 0,\quad m=\overline{0,\infty}.
\end{eqnarray}
The equalities (\ref{ty4}) and inequalities (\ref{ty5}) yield  $g_m^{0}=\tilde g_m^{0}, \ m=\overline{0,\infty},$ or $\tilde g^0=g^0.$

Prove that $G_n=G_0,$ where $G_n$ is a set of non-decreasing processes  $g=\{g_m\}_{m=0}^{\infty}$ such that  $\{f_m+g_m\}_{m=0}^{\infty}$  is a supermartingale relative to all measures from  $M_n.$  Really, if  $g=\{g_m\}_{m=0}^{\infty}$  is a non-decreasing process from  $G_n,$ then it belongs to $G_0,$ owing to that   $M_n \supset \tilde M_n$ and   $G_n \subseteq G_0.$
Suppose that  $g=\{g_m\}_{m=0}^{\infty}, \ g_0=0,$  is a non-decreasing process from  $G_0.$ It means that 
\begin{eqnarray}\label{c28}
 E^Q\{f_m+  g_m|{\cal F}_{k}\}\leq f_{k}+g_{k}, \quad m=\overline{1,\infty},  \quad 0 \leq  k\leq m, \quad Q \in \tilde M_n.
\end{eqnarray} 
The last inequalities can be written in the form 
$$\sum\limits_{i=1}^n\alpha_i\int\limits_{A}(f_m+  g_m)dP_i\leq \sum\limits_{i=1}^n\alpha_i\int\limits_{A}(f_{k}+  g_{k})dP_i, \quad  m=\overline{1,\infty},  \quad  0 \leq  k \leq m, $$ 
$$   A \in {\cal F}_{k}, \quad  \alpha_i>0, \quad i=\overline{1,n}. $$
By passing to the limit, as $\alpha_j \to 0,  \ \alpha_j>0, \ j \neq i, \  \alpha_i \to 1,$ we obtain 
$$\int\limits_{A}(f_m+  g_m)dP_i\leq \int\limits_{A}(f_{k}+  g_{k})dP_i, \quad i=\overline{1,n},  \quad A \in {\cal F}_{k}, \quad  m=\overline{1,\infty},  \quad  0 \leq  k\leq m.$$
The last inequalities yield inequalities
$$\sum\limits_{i=1}^n\alpha_i\int\limits_{A}(f_m+  g_m)dP_i\leq \sum\limits_{i=1}^n\alpha_i\int\limits_{A}(f_{k}+  g_{k})dP_i, \quad  m=\overline{1,\infty}, \quad   0 \leq k\leq m,  $$ $$  A \in {\cal F}_{k},  \quad  \alpha_i \geq 0, \quad i=\overline{1,n},$$
or 
\begin{eqnarray*}
 E^Q\{f_m+  g_m|{\cal F}_{k}\}\leq f_{k}+g_{k}, \quad m=\overline{1,\infty},  \quad  0 \leq k\leq m,  \quad Q \in M_n.
\end{eqnarray*} 
It means that  $g=\{g_m\}_{m=0}^{\infty}$ belongs to  $G_n.$  On the basis of the above proved,  for the maximal element $ \tilde g^0=\{ \tilde g_m^0\}_{m=0}^{\infty}$  in   the maximal chain   $\tilde G_0 \subseteq G_0$ the equalities 
\begin{eqnarray}\label{clobal} 
 E^Q \{f_m+ \tilde g_m^0|{\cal F}_{k}\}=f_{k}+ \tilde g_{k}^0, \quad  m=\overline{1,\infty},  \quad 1 \leq k\leq m,  \quad Q \in \tilde M_n,
\end{eqnarray}
\begin{eqnarray}\label{cchukal} 
 E^Q (f_m+ \tilde g_m^0)=f_0, \quad  m=\overline{1,\infty},   \quad Q \in \tilde M_n,
\end{eqnarray}
are valid.  From proved equality   $G_n=G_0,$  it follows that  $\tilde G_0$  is a maximal chain in  $G_n .$ 

As far as, $G_0$ coincides with $G_n$  we proved that  the maximal element $\tilde g^0$ in a certain maximal chain in  $G_n$  satisfies equalities
\begin{eqnarray}\label{my5}
E^{P_0}\{f_m+\tilde g_m^0|{\cal F}_{k}\} = f_{k}+\tilde g_{k}^0,\quad m=\overline{1,\infty},  \quad 1 \leq k \leq m,\quad P_0 \in  M_n,
\end{eqnarray}
\begin{eqnarray}\label{my6}
E^{P_0}(f_m+\tilde g_m^{0})=f_0 ,\quad  m=\overline{1,\infty},  \quad P_0 \in M_n.
\end{eqnarray}
Due to arbitrariness of the set of measure $P_1, \ldots, P_n,\ P_i \in M,$  the set $G$ contains nonzero element $\tilde g^0$ and in the maximal chain   $\tilde G \subseteq G$ containing element $\tilde g^0$ the  maximal element $g^0=\{g_m^0\}_{m=0}^{\infty}, g_0^0=0,$ coincides with $\tilde g^0.$ The last statement can be proved as in the case of maximal chain $\tilde G_0.$
So,
\begin{eqnarray}\label{my3}
E^{P_0}\{f_m+g_m^0|{\cal F}_{k}\} = f_{k}+g_{k}^0,\quad m=\overline{1,\infty},  \quad 1 \leq k \leq m,\quad P_0 \in M,
\end{eqnarray}
\begin{eqnarray}\label{my4}
E^{P_0}(f_m+g_m^{0})=f_0 ,\quad m=\overline{1,\infty},  \quad P_0 \in M.
\end{eqnarray}
Denote by  $\{\bar M_m, {\cal F}_m\}_{m=0}^{\infty}$ a martingale relative to the set of measures $M,$ where   $\bar M_m=f_m+g_m^{0}, \ m=\overline{1,\infty}.$
Due to  Theorem \ref{ct4} conditions, the supermartingale    $\{f_m, {\cal F}_m\}_{m=0}^\infty$
and non-decreasing process      $g^0=\{g_m^0\}_{m=0}^\infty$ are uniformly integrable relative to any measure from  $M,$  since for the non-decreasing process     $g^0=\{g_m^0\}_{m=0}^\infty$ there hold bounds    $E^Pg_{m}^0< T+f_0, \ m=\overline{1,\infty}, \ P \in M.$
Therefore, the martingale   $\{\bar M_m, {\cal F}_m\}_{m=0}^\infty$ is uniformly integrable  relative to any measure from   $M.$  
So, with probability 1 relative to every measure from   $M$ there exist limits  
\begin{eqnarray*}%\label{cp38}
\lim\limits_{m \to \infty} \bar M_{m}=M_{\infty}= f_{\infty}+ g_{\infty}^0, \quad \lim\limits_{m \to \infty}f_{m} = f_{\infty}, \quad \lim\limits_{m \to \infty}g_{m}^0 = g_{\infty}^0.
\end{eqnarray*}
Moreover, the representation   
\begin{eqnarray}\label{cp39}
\bar  M_m=E^P\{( f_{\infty}+ g_{\infty}^0)| {\cal F}_m\}, \quad m=\overline{1,\infty},   \quad P \in M,
\end{eqnarray}
 holds, where
  $\bar M=\{ \bar M_m\}_{m=0}^{\infty}$  does not depend on  $P \in M.$
 \end{proof}
In the next theorem we give the necessary and sufficient conditions of regularity of supermartingales.
\begin{thm}\label{reww1}
Let  a supermartingale    $\{f_{m},\ {\cal F}_m\}_{m=0}^{\infty}$ relative to a convex set of equivalent measures $M$ satisfy conditions (\ref{buti1}).
 The necessary and sufficient conditions for it to  be a regular one is the existence 
of adapted nonnegative random process $\bar g^0=\{\bar g_m^0\}_{m=0}^\infty, \ \bar g_0^0=0,$ $E^P\bar g_m^0<\infty,\ m=\overline{1, \infty}, \ P \in M, $ such that  equalities
\begin{eqnarray}\label{reww2}
E^P\{f_{m-1} -f_m|{\cal F}_{m-1}\}=E^P\{\bar g_m^0|{\cal F}_{m-1}\}, \quad m=\overline{1, \infty} , \quad P \in M,
\end{eqnarray} 
are valid.
\end{thm}
\begin{proof} 
{ \bf Necessity.} If   $\{f_{m},\ {\cal F}_m\}_{m=0}^{\infty}$ is  a regular supermartingale,  then there exist a martingale  $\{\bar M_{m},\ {\cal F}_m\}_{m=0}^{\infty}$ and a non-decreasing nonnegative random process  $\{g_{m},\ {\cal F}_m\}_{m=0}^{\infty},$ $ \ g_0=0,$ such that
\begin{eqnarray}\label{reww3}
f_m = \bar M_m - g_m, \quad m=\overline{1, \infty}.
\end{eqnarray}
As before, equalities (\ref{reww3}) yield inequalities $E^Pg_m \leq f_0+T, \  m=\overline{1, \infty},$
and  equalities
$$E^P\{f_{m-1} -f_m|{\cal F}_{m-1}\}=$$
\begin{eqnarray}\label{reww4}
=E^P\{ g_m - g_{m-1}|{\cal F}_{m-1}\}=E^P\{\bar g_m^0|{\cal F}_{m-1}\}, \quad m=\overline{1, \infty} , \quad P \in M,
\end{eqnarray}
where we introduced the denotation $\bar g_m^0=g_m - g_{m-1} \geq 0.$
It is evident that $E^P\bar g_m^0\leq 2( f_0+T).$

{\bf Sufficiency.}  If there exists an adapted nonnegative random process $\bar g^0=\{\bar g_m^0\}_{m=0}^\infty, \ \bar g_0^0=0,$  $ E^P\bar g_m^0<\infty, \ m=\overline{1, \infty}, $ such that the equalities (\ref{reww2}) are valid, then let us consider a  random process $\{\bar M_{m},\ {\cal F}_m\}_{m=0}^{\infty},$ where
\begin{eqnarray}\label{reww5}
\bar M_0=f_0, \quad \bar M_m=f_m+\sum\limits_{i=1}^m\bar g_m^0, \quad m=\overline{1, \infty}.
\end{eqnarray}
It is evident that $E^P|\bar M_m|< \infty$ and
$$E^P\{\bar M_{m-1} -  \bar M_m|{\cal F}_{m-1}\}=E^P\{f_{m-1} - f_m- \bar g_m^0|{\cal F}_{m-1}\}=0.$$
Theorem \ref{reww1} is proven.
\end{proof}
In the next Theorem we describe  the structure of non-decreasing process for a regular supermartingale.
\begin{thm}\label{kj1}   
 Let  a  supermartingale $\{f_m, {\cal F}_m\}_{m=0}^\infty$  relative to a convex set of equivalent measures $M$  satisfy conditions (\ref{buti1}). The necessary and sufficient conditions for it to be  regular one is the existence  of a non-decreasing  adapted  process $ g=\{ g_m\}_{m=0}^{\infty}, \ g_0=0, $ and   adapted  processes
$\bar \Psi^{j}=\{ \bar \Psi^{j}_m\}_{m=0}^{\infty}, $ $  \bar \Psi^{j}_0=0, \ j=\overline{1,n},$ such that between elements $g_m, \ m=\overline{1,\infty},$  of non-decreasing process  $g=\{g_m\}_{m=0}^{\infty}$     the relations
\begin{eqnarray}\label{kj2}
g_m -g_{m-1}=f_{m-1}- E^{P_j}\{f_m|{\cal F}_{m-1}\}+  \bar \Psi^{j}_m, \quad  m=\overline{1,\infty}, \quad  j=\overline{1,n},
\end{eqnarray}
 are valid for each  set of measures $P_1, \ldots,P_n \in M$, where  $E^{P_j}|\bar \Psi^{j}_m|< \infty, $  $E^{P_j}\{\bar \Psi^{j}_m|{\cal F}_{m-1}\}=0, \ j=\overline{1,n}, \ m=\overline{1,\infty}. $  
\end{thm}
\begin{proof} {\bf The necessity.} 
Let $\{f_m, {\cal F}_m\}_{m=0}^\infty$ be a regular supermartingale. Then for it
 the representation 
\begin{eqnarray}\label{kkj3}
 f_m+g_m=M_m, \quad  m=\overline{1,\infty}, \quad  j=\overline{1,n},
\end{eqnarray}
is valid,  where $ \{g_m\}_{m=0}^\infty, \ g_0=0,$ is a non-decreasing adapted process, \\ $\{M_m, {\cal F}_m\}_{m=0}^\infty$ is a martingale relative to the set of measures $M.$
 For any finite set of measures $P_1, \ldots,P_n \in M,$ we have
\begin{eqnarray}\label{kj3}
 E^{P_j}\{f_m+g_m|{\cal F}_{m-1}\}=f_{m-1}+g_{m-1}, \quad  m=\overline{1,\infty}, \quad  j=\overline{1,n}.
\end{eqnarray}
Hence, we have
\begin{eqnarray}\label{kj4}
 E^{P_j}\{g_m- g_{m-1}|{\cal F}_{m-1}\}=f_{m-1}- E^{P_j}\{f_m|{\cal F}_{m-1}\} , \  \ m=\overline{1,\infty}, \ \  j=\overline{1,n}.
\end{eqnarray}
Let us put
\begin{eqnarray}\label{kj5}
 \bar \Psi^{j}_m=g_m - g_{m-1} - E^{P_j}\{g_m - g_{m-1}|{\cal F}_{m-1}\}.
\end{eqnarray}
The assumptions of  Theorem \ref{kj1} and  Lemma \ref{l3},  the representation (\ref{kj5}) imply 
$E^{P_j}|\bar \Psi^{j}_m|< 4(f_0+T), $  $E^{P_j}\{\bar \Psi^{j}_m|{\cal F}_{m-1}\}=0, \ j=\overline{1,n}, \ m=\overline{1,\infty}. $  
This proves  the necessity.

 {\bf The sufficiency.}  For any set of measures $P_1, \ldots,P_n \in M$ the representation (\ref{kj2}) for a non-decreasing adapted  process  $g=\{g_m\}_{m=0}^{\infty}, \ g_0=0,$   is valid. Hence, we obtain (\ref{kj4})    and (\ref{kj3}). The  equalities (\ref{kj4}), (\ref{kj3}) and the formula
\begin{eqnarray*}\label{nsal}
 E^{P}\{f_{m}+g_ m|{\cal F}_{m-1}\}=\frac{\sum\limits_{i=1}^{n}\alpha_i E^{P_1}\{\varphi_i|{\cal F}_{m-1}\}E^{P_i}\{f_{m}+g_m|{\cal F}_{m-1}\}}{\sum\limits_{i=1}^{n}\alpha_i E^{P_1}\{\varphi_i|{\cal F}_{m-1}\}}, \quad P \in M_n,
\end{eqnarray*}
$$\varphi_i=\frac{dP_i}{dP_1}, \quad i=\overline{1,n}, $$
imply
$$ E^{P}\{f_m+g_m|{\cal F}_{m-1}\}=f_{m-1}+g_{m-1}, \quad  m=\overline{1,\infty}, \quad  P \in M_n.$$ 
Arbitrariness of the set of measures $P_1, \ldots,P_n \in M$  and  fulfilment of the condition (\ref{buti1}) for the supermartingale $\{f_m, {\cal F}_m\}_{m=0}^\infty$  imply its   regularity.
\end{proof}

Further, we consider a class of supermartingales $ F$ satisfying conditions

\begin{eqnarray*}\label{f1}
\sup\limits_{P\in M}E^P|f_m|< \infty, \quad m=\overline{0,\infty}. 
\end{eqnarray*}
\begin{defin} A  supermartingale $f=\{f_{m},\ {\cal F}_m\}_{m=0}^{\infty}  \in F$  is said to be    local  regular one if there exists an increasing  sequence of  nonrandom stopping times $\tau_{k_s} = k_s, \  k_s < \infty, \  s=\overline{1, \infty}, \  \lim\limits_{s \to \infty}k_s=\infty,  $ such that the stopped process  $f^{\tau_{k_s}}=\{f_{m\wedge \tau_{k_s}},\ {\cal F}_m\}_{m=0}^{\infty}$
 is a regular supermartingale for every $\tau_{k_s} = k_s, \  k_s < \infty, \  s=\overline{1, \infty}.$
\end{defin}
\begin{thm}\label{hf1}
Let $\{f_m, {\cal F}_m\}_{m=0}^\infty$ be a supermartingale relative to  a convex set of equivalent measures $M,$ belonging to the class $F,$
for which the representation
\begin{eqnarray}\label{jk1}
 f_m=M_m -g_m^0,  \quad m=\overline{0,\infty},
\end{eqnarray} 
 is valid, where $\{M_m\}_{m=0}^\infty$ is a martingale  relative to a convex set of equivalent measures $M$ such that
$$E^P|M_m|< \infty, \quad m=\overline{0,\infty}, \quad P \in M,  $$
 $g^0=\{g_m^0\}_{m=0}^\infty, \  g_0^0=0,$ is a non-decreasing adapted process. Then $\{f_m, {\cal F}_m\}_{m=0}^\infty$ is a  local regular supermartingale.
\end{thm}
\begin{proof}
The representation (\ref{jk1})  and assumptions of  Theorem \ref{hf1} imply inequalities
$ E^Pg_m^0<\infty, \ m=\overline{1,\infty}, \ P \in M.$
For any  measure $P \in M,$ therefore we have
\begin{eqnarray}\label{jk2}
 E^{P}\{f_m+g_m^0|{\cal F}_{m-1}\}=M_{m-1}=f_{m-1}+g_{m-1}^0, \quad  m=\overline{1,\infty}.
\end{eqnarray}
Consider a sequence of stopping times $\tau_s=s, \ s=\overline{1, \infty}.$
Equalities (\ref{jk2}) yield
\begin{eqnarray}\label{mif1}
 E^{P}\{f_{m\wedge\tau_s} +g_{m\wedge\tau_s}^0|{\cal F}_{m-1}\}=M_{(m-1)\wedge\tau_s}=
f_{(m-1)\wedge\tau_s}+g_{(m-1)\wedge\tau_s}^0,
\end{eqnarray}
$$  m=\overline{1,\infty}, \quad P \in M. $$
For the stopped supermartingale   $\{f_{m\wedge\tau_s}, {\cal F}_m\}_{m=0}^\infty,$
the set $G$ of adapted non-decreasing processes $g=\{g_m\}_{m=0}^\infty, \  g_0=0,$  such that   $\{f_{m\wedge\tau_s}+g_m, {\cal F}_m\}_{m=0}^\infty$
 is a supermartingale  relative to a convex set of equivalent measures $M$
contains nonzero element $g^{0, \tau_s}=$  $\{g_{m\wedge\tau_s}^0\}_{m=0}^\infty,$ $  g_0^0=0.$
Consider a maximal chain $\tilde G \subseteq G$  containing   this element  and let $g=\{g_m\}_{m=0}^\infty, \  g_0=0,$ be a maximal element in $\tilde G$ which exists, since the stopped supermartingale   $\{f_{m\wedge\tau_s}, {\cal F}_m\}_{m=0}^\infty$ is such that $|f_{m\wedge\tau_s}| \leq \sum\limits_{i=0}^s|f_i|=\varphi, \  m=\overline{0,\infty}, \  E^P\varphi \leq \sum\limits_{i=0}^s\sup\limits_{P \in M}E^P|f_i|=T< \infty.$
 Then 
\begin{eqnarray}\label{jk3}
 E^{P}\{f_{m\wedge\tau_s}+g_m|{\cal F}_{m-1}\}\leq f_{(m-1)\wedge\tau_s}+g_{m-1}, \quad  m=\overline{1,\infty}.
\end{eqnarray}
Equalities (\ref{mif1}) and inequality $ g^{0, \tau_s} \preceq g$ imply
\begin{eqnarray}\label{jk4}
f_0 =  E^{P}\{f_{m\wedge\tau_s}+g_{m\wedge\tau_s}^0\}\leq E^{P}\{f_{m\wedge\tau_s}+g_m\} \leq f_0, \  m=\overline{1,\infty}, \ P \in M.
\end{eqnarray}
The last  inequalities yield 
\begin{eqnarray}\label{jk4}
 E^{P}\{f_{m\wedge\tau_s}+g_m\} = f_0, \quad  m=\overline{1,\infty}, \quad  P \in M.
\end{eqnarray}
 The equalities (\ref{jk4}), inequality  $  g^{0, \tau_s} \preceq g,$ and equalities
\begin{eqnarray}\label{jk5}
 E^{P}\{f_{m\wedge\tau_s}+g_{m\wedge\tau_s}^0\}=M_{0}=f_{0}, \quad  m=\overline{1,\infty}, \quad  P \in M,
\end{eqnarray}
imply that $ g^{0, \tau_s}=g.$   

So, we proved that the stopped supermartingale
 $\{f_{m\wedge\tau_s}, {\cal F}_m\}_{m=0}^\infty$ is regular one for every stopping time $\tau_s, \  s=\overline{1, \infty},$  converging to  the infinity,
as $ s \to \infty.$ This proves Theorem \ref{hf1}.
\end{proof}

\begin{thm}\label{mars1}
Let a supermartingale $\{ f_m, {\cal F}_m\}_{m=0}^\infty $ relative to  a convex set of equivalent measures $M$ on  a measurable space $\{\Omega, {\cal F}\}$ belongs to a class $F$   and there exists a nonnegative adapted  random process $\{\bar  g_m^0\}_{m=1}^\infty, \ E^P \bar g_m^0< \infty, \  m=\overline{1, \infty}, \  P \in M,$ such that 
\begin{eqnarray}\label{mars2}
f_{m-1} - E^P\{f_m|{\cal F}_{m-1}\}=E^P\{\bar g_m^0|{\cal F}_{m-1}\}, \quad m=\overline{1, \infty}, \quad P\in M,
\end{eqnarray}
then $\{ f_m, {\cal F}_m\}_{m=0}^\infty $ is a local regular supermartingale.
\end{thm}
\begin{proof} To prove Theorem \ref{mars1} let us consider a random process
$$\bar M_m=f_m+\sum\limits_{i=1}^m \bar g_i^0, \quad m=\overline{1, \infty}, \quad P\in M, \quad  f_0=\bar M_0.$$
 It is evident that
$E^P|\bar M_m | < \infty, \ m=\overline{1,\infty}, \ P \in M,$ and $E^P\{\bar M_m|{\cal F}_{m-1}\}=\bar M_{m-1}, \ m=\overline{1, \infty}, \ P\in M.$
Therefore, for $f_m$ the representation 
\begin{eqnarray}\label{mars3}
f_{m} =\bar M_m - g_{m}, \quad m=\overline{0, \infty}, 
\end{eqnarray}
is valid, where  $g_m=\sum\limits_{i=1}^m\bar g_i^0.$
Supermartingale (\ref{mars3}) satisfies conditions of the Theorem 6.
The Theorem  \ref{mars1} is proved.
\end{proof}

Below we describe local regular supermartingales. For this we need some auxiliary statements. Denote by $N_0=[1, 2, \ldots, \infty)$ the set of positive natural numbers.

On a measurable space $\{\Omega, {\cal F}\}$ let us  consider two sub $\sigma$-algebras   $G_n \subset G_N$ of $\sigma$-algebra ${\cal F}.$ 
We suppose that for $N >n$  $\sigma$-algebra  $G_N$ is generated by  sets $E_s,$ $s=\overline{1,\infty},$  satisfying conditions $E_j \cap E_m=\emptyset, \  j \neq m, $ $ \bigcup\limits_{s=1}^\infty E_s=\Omega.$
We assume that
 $G_n$ is generated by  sets $F_j,$ $j=\overline{1,\infty},$ satisfying conditions $F_j \cap F_m=\emptyset, \  j\neq m, \ \bigcup\limits_{j=1}^\infty F_j=\Omega,$ 
 and such that  $F_j=\bigcup\limits_{s \in I_j} E_s, \ j=\overline{1,\infty},$ where $I_j $ are subsets of the set $N_0, $  $I_r \cap I_l =\emptyset, \ r\neq l,$ $\bigcup\limits_{j=1}^\infty I_j=N_0.$

\begin{lemma}\label{mq1} 
Let  $P_1, \ldots, P_k$ be a set of  equivalent measures on a measurable space $\{\Omega, {\cal F}\}.$  If $P_1(E_s)>0, \ s=\overline{1,\infty},$ then the formulas 
\begin{eqnarray}\label{mn1}
\frac{E^{P_l}\{\frac{dP_i}{dP_l}| G_N\}}{ E^{P_l}\{\frac{dP_i}{dP_l}| G_n\} }=
\sum\limits_{j=1}^\infty \sum\limits_{s \in I_j} \frac{P_i(E_s) P_l(F_j)}{P_i(F_j)P_l(E_s)} \chi_{E_s}(\omega), \quad l=\overline{1, k},
\end{eqnarray}
are valid.
\end{lemma}
\begin{proof}
It is evident that
\begin{eqnarray}\label{mn2}
E^{P_l}\left\{\frac{dP_i}{dP_l}| G_N\right\}=\sum\limits_{s=1}^\infty\frac{1}{P_l(E_s)}\int\limits_{E_s}\frac{dP_i}{dP_l}dP_l \chi_{E_s}(\omega)=\sum\limits_{s=1}^\infty\frac{P_i(E_s)}{P_l(E_s)}\chi_{E_s}(\omega),
\end{eqnarray}
\begin{eqnarray}\label{mn3}
E^{P_l}\left\{\frac{dP_i}{dP_l}| G_n\right\}=\sum\limits_{j=1}^\infty\frac{P_i(F_j)}{P_l(F_j)}\chi_{F_j}(\omega).
\end{eqnarray}
Since $\chi_{F_j}(\omega)=\sum\limits_{s \in I_j}\chi_{E_s}(\omega)$ we have
\begin{eqnarray}\label{mn4}
E^{P_l}\left\{\frac{dP_i}{dP_l}| G_n\right\}=\sum\limits_{j=1}^\infty \sum\limits _{s \in I_j}\frac{P_i(F_j)}{P_l(F_j)}\chi_{E_s}(\omega).
\end{eqnarray}
Therefore,
\begin{eqnarray*}\label{mn5}
\frac{E^{P_l}\left\{\frac{dP_i}{dP_l}| G_N\right\}}{ E^{P_l}\left\{\frac{dP_i}{dP_l}| G_n\right\}}=\frac{\sum\limits_{s=1}^\infty\frac{P_i(E_s)}{P_l(E_s)}\chi_{E_s}(\omega)}{\sum\limits_{j=1}^\infty \sum\limits _{s \in I_j}\frac{P_i(F_j)}{P_l(F_j)}\chi_{E_s}(\omega)}=\frac{\sum\limits_{j=1}^\infty\sum\limits _{s \in I_j}\frac{P_i(E_s)}{P_l(E_s)}\chi_{E_s}(\omega)}{\sum\limits_{j=1}^\infty \sum\limits _{s \in I_j}\frac{P_i(F_j)}{P_l(F_j)}\chi_{E_s}(\omega)}=
\end{eqnarray*}
\begin{eqnarray}\label{mn6}
\sum\limits_{j=1}^\infty \sum\limits _{s \in I_j}\frac{P_i(E_s)}{P_i(F_j)}\frac{P_l(F_j)}{P_l(E_s)}\chi_{E_s}(\omega).
\end{eqnarray}
The Lemma \ref{mq1} is proved.
\end{proof}
 
\begin{lemma}\label{mq7}  Let a set of equivalent measures $P_1, \ldots, P_k$ on $\{\Omega, {\cal F}\} $ are such that for a certain  $1 \leq i_0 \leq k$ 
 conditional measures $ \frac{P_i(A_s)}{P_i(F_j)},  A_s  \subseteq F_j , j=\overline{1, \infty}, \ i=\overline{1,k}, $ satisfy conditions
\begin{eqnarray}\label{mn8}
\frac{P_i(A_s)}{P_i(F_j)} \leq \frac{P_{i_0}(A_s)}{P_{i_0}(F_j)}, \quad  A_s  \subseteq F_j, \quad \bigcup\limits_{s \in I_j}A_s= F_j, \quad j=\overline{1, \infty},  \quad  i=\overline{1,k}.
\end{eqnarray}
Then the inequalities
\begin{eqnarray}\label{mn9}
\frac{E^{P_l}\{\frac{dP_i}{dP_l}| G_N\}}{ E^{P_l}\{\frac{dP_i}{dP_l}| G_n\} } \leq
\frac{E^{P_l}\{\frac{dP_{i_0}}{dP_l}| G_N\}}{ E^{P_l}\{\frac{dP_{i_0}}{dP_l}| G_n\} }, \quad i=\overline{1,k}, \quad  l=\overline{1,k},
\end{eqnarray}
are valid.
\end{lemma}
\begin{proof} The proof of the Lemma  \ref{mq7} follows from the formulas (\ref{mn6}).
\end{proof} 
\begin{defin}\label{v1} A filtration ${\cal F}_n \subset {\cal F}_{n+1}, \ n=\overline{1, \infty},$ on a  measurable space $\{\Omega, {\cal F}\}$ satisfies {\bf condition} ${\bf A}, $ if \\
1)  $\sigma$-algebra ${\cal F}$ coinsides with minimal $\sigma$-algebra  generated by   the sets belonging to the set $\bigcup\limits_{n=0}^\infty {\cal F}_n;$ \\
2) ${\cal F}_n$ is generated by sets $A_s^n \subset  {\cal F}, \ s=\overline{1, \infty}, \ n=\overline{1, \infty},$ such that 
 \begin{eqnarray*} A_m^n\cap A_j^n=\emptyset, \quad  m \neq j, \quad \bigcup\limits_{s=1}^ \infty A_s^n=\Omega, \quad  A_s^n=\bigcup\limits_{j \in I_s^n}A_j^{n+1}, \quad   s=\overline{1, \infty}, 
\end{eqnarray*} 
 \begin{eqnarray*}
  I_s^n \cap I_m^n=\emptyset, \quad  s \neq m, \quad \bigcup\limits_{s=1}^ \infty I_s^n=N_0, \quad   n=\overline{1, \infty}.
\end{eqnarray*}
\end{defin}
\begin{defin}\label{v2}
 On a measurable space $\{\Omega, {\cal F}\}$ with filtration ${\cal F}_n$   satisfying {\bf condition} ${\bf A}$   a set of equivalent measures  $P_1, \ldots, P_k$ satisfies  {\bf condition} ${\bf B}$ if 
\begin{eqnarray*}
P_1(A_s^n) >0,  \quad  s=\overline{1, \infty}, \quad  n=\overline{1, \infty},
\end{eqnarray*}
and for a certain $1 \leq i_0 \leq k$ the  inequalities 
\begin{eqnarray*}
\frac{P_i(A_j^{n+1})}{P_i(A_s^n)} \leq \frac{P_{i_0}(A_j^{n+1})}{P_{i_0}(A_s^n)}, \quad j \in I_s^n, \quad n=\overline{1, \infty},
\end{eqnarray*}
are valid.
\end{defin}

\begin{lemma}\label{mq10}
Let a filtration ${\cal F}_n$ and a set of equivalent measures $P_1, \ldots, P_k$ on a measurable space $\{\Omega, {\cal F}\} $ satisfy conditions $A$ and
   $B,$ correspondingly.
 Then for every $1\leq l \leq k$ and $1 \leq n \leq \infty$ the inequalities 
\begin{eqnarray}\label{mn13}
\frac{\frac{dP_i}{dP_l}}{ E^{P_l}\{\frac{dP_i}{dP_l}|  {\cal F}_n\} } \leq
\frac{\frac{dP_{i_0}}{dP_l}}{ E^{P_l}\{\frac{dP_{i_0}}{dP_l}|  {\cal F}_n\} }, \quad i=\overline{1,k}, \quad l=\overline{1,k},
\end{eqnarray}
 are valid. 
\end{lemma}
\begin{proof} Taking into account  Lemma \ref{mq7} 
  for every $1\leq l \leq k$ and $N \geq n \geq 1$  we obtain the inequalities 
\begin{eqnarray}\label{mn12}
\frac{E^{P_l}\{\frac{dP_i}{dP_l}| {\cal F}_N\}}{ E^{P_l}\{\frac{dP_i}{dP_l}|  {\cal F}_n\} } \leq
\frac{E^{P_l}\{\frac{dP_{i_0}}{dP_l}|  {\cal F}_N\}}{ E^{P_l}\{\frac{dP_{i_0}}{dP_l}|  {\cal F}_n\} }, \quad i=\overline{1,k}, \quad l=\overline{1,k}.
\end{eqnarray}
 Since  a random value $\frac{dP_i}{dP_l}$ is measurable one relative to the $\sigma$-algebra ${\cal F}$ and  integrable with respect to the measure $P_l,$ then  the conditions of Levy Theorem are valid. It implies 
 that with probability 1 $\lim\limits_{N \to \infty}E^{P_l}\{\frac{dP_i}{dP_l}|{\cal F}_N\}=\frac{dP_i}{dP_l}. $  Passing to the limit in the inequalities   (\ref{mn12}), as $N \to \infty,$ we obtain  the inequalities   
(\ref{mn13}) and the  proof of
Lemma \ref{mq10}.
\end{proof}

Let $P_1, \ldots, P_k $ be a  family of  equivalent  measures on a measurable space   $\{ \Omega, {\cal F}\}$ and let us introduce denotation
 $$M=\left\{Q, \  Q=\sum\limits_{i=1}^{k}\alpha_i P_i, \ \alpha_i \geq 0, \  i=\overline{1,k},\ \sum\limits_{i=1}^{k}\alpha_i=1\right\}.$$

\begin{lemma}\label{q2}
If  $\xi$  is an  integrable random value relative to the set of equivalent  measures $P_1, \ldots, P_k $, then the formula
\begin{eqnarray}\label{n2}
\sup\limits_{Q \in M}E^{Q}\{\xi|{\cal F}_n\}=\max\limits_{1 \leq i \leq k}E^{P_i}\{\xi|{\cal F}_n \}
\end{eqnarray}
is valid almost everywhere relative to the  measure $P_1$.
\end{lemma}
\begin{proof} Using the formula
\begin{eqnarray}\label{n3}
 E^{Q}\{\xi|{\cal F}_n\}=\frac{\sum\limits_{i=1}^{k}\alpha_i E^{P_1}\{\varphi_i|{\cal F}_n\}E^{P_i}\{\xi|{\cal F}_n\}}{\sum\limits_{i=1}^{k}\alpha_i E^{P_1}\{\varphi_i|{\cal F}_n\}}, \quad Q \in M,
\end{eqnarray}
 where $\varphi_i=\frac{dP_i}{dP_1},$ we obtain the inequality 
$$ E^{Q}\{\xi|{\cal F}_n\} \leq \max\limits_{1 \leq i \leq k}E^{P_i}\{\xi|{\cal F}_n\},$$  
or 
$$\sup\limits_{Q \in M} E^{Q}\{\xi|{\cal F}_n\} \leq \max\limits_{1 \leq i \leq k}E^{P_i}\{\xi|{\cal F}_n\}.$$
On the other side  
$$ E^{P_i}\{\xi|{\cal F}_n\} \leq \sup\limits_{Q \in M} E^{Q}\{\xi|{\cal F}_n\} .$$ 
Therefore,
$$\max\limits_{1 \leq i \leq k} E^{P_i}\{\xi|{\cal F}_n\} \leq \sup\limits_{Q \in M} E^{Q}\{\xi|{\cal F}_n\} .$$ 
The Lemma \ref{q2} is proved.
\end{proof}

\begin{lemma}\label{hj1}
Let $G$ be a sub $\sigma$-algebra of ${\cal F}$ and $f_1, \ldots, f_n$ be nonnegative integrable random values relative to every measure from $M.$ Then 
\begin{eqnarray}\label{rgps1}
 E^P\{\max\{f_1, \ldots, f_n\}|G\}\geq \max\{E^P\{f_1|G\}, \ldots, E^P\{f_n|G\}\}, \quad P \in M. \end{eqnarray}
\end{lemma}
\begin{proof} From inequalities
\begin{eqnarray}\label{gps1}
\max\limits_{1 \leq  i \leq n} f_i \geq f_j , \quad j=\overline{1,n},
\end{eqnarray}
we have
\begin{eqnarray}\label{gps2}
E^P\{\max\limits_{1 \leq  i \leq n} f_i|G\}\geq  E^P\{ f_j|G\} , \quad j=\overline{1,n}.
\end{eqnarray}
The last imply
\begin{eqnarray}\label{gps3}
E^P\{\max\limits_{1 \leq  i \leq n} f_i|G\}  \geq  \max\limits_{1 \leq  i \leq n} E^P\{f_i|G\}.
\end{eqnarray}
\end{proof}

In the next Lemma we present formula for calculation of conditional expectatation relative to another measure from $M.$
\begin{lemma}\label{q1}  
Let $M$ be a convex  set of equivalent measures  and let  $\eta$ be an integrable random value relative to every measure from $M$ on a measurable space  $\{ \Omega, {\cal F}\}.$  Then the following formula 
\begin{eqnarray}\label{n1}
E^{P_1}\{\eta|{\cal F}_n\}=E^{P_2}\left\{\eta \varphi_n^{P_1}|{\cal F}_n\right\}, \quad n=\overline{1, \infty},   
\end{eqnarray}
is valid,  where
\begin{eqnarray*}
 \varphi_n^{P_1}=\frac{dP_1}{dP_2}\left[E^{P_2}\left\{\frac{dP_1}{dP_2}|{\cal F}_n\right\}\right]^{-1}, \quad P_1, \ P_2 \in M.
\end{eqnarray*}
\end{lemma}
\begin{proof}The proof of Lemma \ref{q1}   is evident.
\end{proof}

\begin{lemma}\label{q3} Suppose that a filtration ${\cal F}_n$  and a set of equivalent measures $\{P_1, \ldots, P_k \}$ on $\{ \Omega, {\cal F}\}$ satisfy conditions $A$ and $B,$ correspondingly. Let $\xi$ be a nonnegative bounded random value on  a measurable space $\{ \Omega, {\cal F}\}.$ 
 Then the formulae 
\begin{eqnarray}\label{n3}
 E^{P_l}\{\max\limits_{1 \leq i \leq k} E^{P_i}\{\xi|{\cal F}_n\}|{\cal F}_m\}=
 \max\limits_{1 \leq i \leq k} E^{P_l}\{\xi  \varphi_n^{P_i}|{\cal F}_m\}, \quad n>m,  \quad l=\overline{1,k},
\end{eqnarray}
are valid, where
$$ \varphi_n^{P_i}=\frac{dP_i}{dP_l}\left[E^{P_l}\left\{\frac{dP_i}{dP_l}|{\cal F}_n\right\}\right]^{-1}.$$
\end{lemma}
\begin{proof}  From Lemma \ref{q1} we obtain
$$\max\limits_{1 \leq i \leq k}E^{P_i}\{\xi|{\cal F}_n\}= \max\limits_{1 \leq i \leq k}E^{P_l}\{\xi  \varphi_n^{P_i}|{\cal F}_n\}, \quad l=\overline{1,k}.$$
Let us introduce denotation $T_i= \xi \varphi_n^{P_i}.$ Then  $T_i$ is an integrable random value and
$$ \max\limits_{1 \leq i \leq k}E^{P_i}\{\xi|{\cal F}_n\}=\max\limits_{1 \leq i \leq k}E^{P_l}\{\xi  \varphi_n^{P_i}|{\cal F}_n\}=\max\limits_{1 \leq i \leq k}E^{P_l}\{T_i|{\cal F}_n\}, \quad l=\overline{1,k}.$$
Due to Lemma \ref{hj1} we obtain the inequality
$$E^{P_l}\{ \max\limits_{1 \leq i \leq k}E^{P_i}\{\xi|{\cal F}_n\}|{\cal F}_m\}=
E^{P_l}\{\max\limits_{1 \leq i \leq k}E^{P_l}\{T_i|{\cal F}_n\}|{\cal F}_m\} \geq $$
$$\max\limits_{1 \leq i \leq k}E^{P_l}\{E^{P_l}\{T_i|{\cal F}_n\}|{\cal F}_m\}=
\max\limits_{1 \leq i \leq k}E^{P_l}\{T_i|{\cal F}_m\}.$$
Let us prove reciprocal inequality
$$E^{P_l}\{ \max\limits_{1 \leq i \leq k}E^{P_i}\{\xi|{\cal F}_n\}|{\cal F}_m\} \leq
\max\limits_{1 \leq i \leq k}E^{P_l}\{T_i|{\cal F}_m\}.$$
The last inequality follows from the fact that  $\max\limits_{1 \leq i \leq k}E^{P_l}\{T_i|{\cal F}_n\}=$\\ $E^{P_{l}}\{T_{i_0}|{\cal F}_n\}.$
Really,
$$E^{P_l}\{ \max\limits_{1 \leq i \leq k}E^{P_i}\{\xi|{\cal F}_n\}|{\cal F}_m\}= 
E^{P_l}\{E^{P_{l}}\{T_{i_0}|{\cal F}_n\}|{\cal F}_m\}=E^{P_{l}}\{T_{i_0}|{\cal F}_m\}\leq $$
$$ \max\limits_{1 \leq i\leq k}E^{P_l}\{T_{i}|{\cal F}_m\}.$$
Lemma \ref{q3} is proved.
\end{proof}

The next Lemma is a consequence of Lemma \ref{q3}.
\begin{lemma}\label{qq3} Let a filtration ${\cal F}_n$  and the set of equivalent measures $\{P_1, \ldots, P_k \}$ on a measurable space $\{ \Omega, {\cal F}\}$ satisfy conditions $A$ and $B,$ correspondingly  and let $\xi$ be a nonnegative bounded random value 
on $\{ \Omega, {\cal F}\}.$ 
 Then the equalities
\begin{eqnarray}\label{nn3}
 E^{P_l}\{\xi \max\limits_{1 \leq i \leq k}    \varphi_n^{P_i}|{\cal F}_n\}=
 \max\limits_{1 \leq i \leq k} E^{P_l}\{\xi  \varphi_n^{P_i}|{\cal F}_n\}, \quad l=\overline{1,k}, \quad n=\overline{0,\infty},
\end{eqnarray}
are valid, where
$$ \varphi_n^{P_i}=\frac{dP_i}{dP_l}\left[E^{P_l}\left\{\frac{dP_i}{dP_l}|{\cal F}_n\right\}\right]^{-1}.$$
\end{lemma}
\begin{proof}
\begin{eqnarray*}\label{nn3}
 \max\limits_{1 \leq i \leq k}E^{P_l}\{ \xi  \varphi_n^{P_i}|{\cal F}_n\} \leq 
 E^{P_l}\{ \xi \max\limits_{1 \leq i \leq k}  \varphi_n^{P_i}|{\cal F}_n\}\leq
\end{eqnarray*}
\begin{eqnarray}\label{nn4}
  E^{P_l}\{\xi  \varphi_n^{P_{i_0}}|{\cal F}_n\} \leq \max\limits_{1 \leq i \leq k}E^{P_l}\{\xi  \varphi_n^{P_i}|{\cal F}_n\} , \quad l=\overline{1,k}, \quad n=\overline{0,\infty}.
\end{eqnarray}
The last inequalities prove Lemma \ref{qq3}.
\end{proof}

\begin{lemma}\label{lkq4} 
 Let a filtration ${\cal F}_n$  and a set of equivalent measures $\{P_1, \ldots, P_k \}$ on  a measurable space $\{ \Omega, {\cal F}\}$ satisfy conditions $A$ and $B,$ correspondingly.
Then for every nonnegative integrable random value $\xi$ relative to the set of measures  $P_1, \ldots, P_k$ the inequalities 
\begin{eqnarray}\label{lkn4}
 E^{P_l}\{\max\limits_{1 \leq i \leq k} E^{P_i}\{\xi|{\cal F}_n\}|{\cal F}_m\}\leq
\max\limits_{1 \leq i \leq k}  E^{P_i}\{\xi |{\cal F}_m\}, \quad n>m,  \quad l=\overline{1,k},
\end{eqnarray}
are  valid.
\end{lemma}
\begin{proof}
First, consider the case of bounded nonnegative random value $\xi.$ It is evident that the following equalities 
\begin{eqnarray}\label{por1}
 \bigcup\limits_{i=1}^k\left\{\omega, \ E^{P_l}\left\{\frac{dP_i}{dP_l}|{\cal F}_n\right\}\geq E^{P_l}\left\{\frac{dP_i}{dP_l}|{\cal F}_m\right\}\right\} =\Omega,  \ n > m, 
\end{eqnarray}
are valid.
Due to  (\ref{por1}) for every $\omega \in \Omega$ there exist  $1 \leq i \leq k $ such that
\begin{eqnarray}\label{por3}
\frac{\xi \frac{dP_i}{dP_l}}{E^P\{\frac{dP_i}{dP_l}|{\cal F}_n\}} \leq \frac{\xi \frac{dP_i}{dP_l}}{E^P\{\frac{dP_i}{dP_l}|{\cal F}_m\}}.
\end{eqnarray}
Therefore,
\begin{eqnarray}\label{ppor4}
\max\limits_{1 \leq i \leq k}\frac{\xi \frac{dP_i}{dP_l}}{E^{P_l}\{\frac{dP_i}{dP_l}|{\cal F}_n\}} \leq \max\limits_{1 \leq i \leq k}\frac{\xi \frac{dP_i}{dP_l}}{E^{P_l}\{\frac{dP_i}{dP_l}|{\cal F}_m\}}.
\end{eqnarray}
From (\ref{ppor4}) we obtain inequality
\begin{eqnarray}\label{por4}
E^{P_l}\left\{\max\limits_{1 \leq i \leq k}\frac{\xi \frac{dP_i}{dP_l}}{E^{P_l}\{\frac{dP_i}{dP_l}|{\cal F}_n\}}|{\cal F}_m\right\} \leq E^{P_l}\left\{\max\limits_{1 \leq i \leq k}\frac{\xi \frac{dP_i}{dP_l}}{E^{P_l}\{\frac{dP_i}{dP_l}|{\cal F}_m\}}|{\cal F}_m\right\}.
\end{eqnarray}
The Lemmas  \ref{q3},  \ref{qq3} and  inequality (\ref{por4}) prove  Lemma \ref{lkq4}, as $\xi$ is bounded random value. Let us consider the case as 
$\max\limits_{1 \leq i \leq k}  E^{P_i}\xi < \infty.$ Let $\xi_s, s=\overline{1, \infty},$ be a sequence of bounded random values  converging to $\xi$ monotonuosly. Then 
\begin{eqnarray}\label{alkn44}
 E^{P_l}\{\max\limits_{1 \leq i \leq k} E^{P_i}\{\xi_s|{\cal F}_n\}|{\cal F}_m\}\leq
\max\limits_{1 \leq i \leq k}  E^{P_i}\{\xi_s |{\cal F}_m\}, \quad l=\overline{1,k}.
\end{eqnarray}
Due to monotony convergence  of $\xi_s$ to $\xi,$ as $s \to \infty,$ we can pass to the limit under conditional expectations on the left and on the right in inequalities (\ref{alkn44}) that proves  Lemma \ref{lkq4}.  
\end{proof}

\begin{lemma}\label{q5}
 Let a filtration ${\cal F}_n$  and a set of equivalent measures $\{P_1, \ldots, P_k \}$ on a measurable space  $\{ \Omega, {\cal F}\}$ satisfy conditions $A$ and $B,$ correspondingly, and
 let  $\xi$ be an  integrable random value relative to the set of  equivalent measures $P_1, \ldots, P_k.$  Then the inequalities
$$ E^Q\{\sup\limits_{P \in M} E^P\{\xi|{\cal F}_n\}|{\cal F}_m\} \leq 
\sup\limits_{P \in M} E^P\{\xi|{\cal F}_m\}, \quad n>m, \quad Q \in M,$$
are valid.
\end{lemma}
\begin{proof} From the equality
$$\sup\limits_{Q \in M}E^{Q}\{\xi|{\cal F}_n\}=\max\limits_{1 \leq i \leq k}E^{P_i}\{\xi|{\cal F}_n\}$$
we obtain inequality
$$ E^{Q}\left\{\max\limits_{1 \leq i \leq k}E^{P_i}\{\xi|{\cal F}_n\}|{\cal F}_m\right\}=\frac{\sum\limits_{j=1}^{k}\alpha_j E^{P_1}\{\varphi_j|{\cal F}_m\}E^{P_j}\left\{\max\limits_{1 \leq i \leq k}E^{P_i}\{\xi|{\cal F}_n\}|{\cal F}_m\right\}}{\sum\limits_{j=1}^{k}\alpha_j E^{P_1}\{\varphi_j|{\cal F}_m\}}\leq$$
$$\leq \max\limits_{1 \leq i \leq k}E^{P_i}\{\xi|{\cal F}_m\}=\sup\limits_{P \in M} E^P\{\xi|{\cal F}_m\}.$$
Lemma \ref{q5} is proved.
\end{proof}

\begin{lemma}\label{1q5}
 Let a filtration ${\cal F}_n$  and a set of equivalent measures $\{P_1, \ldots, P_k \}$ on a measurable space  $\{ \Omega, {\cal F}\}$ satisfy conditions $A$ and $B,$ correspondingly,
  and  let  $\xi$ be a nonnegative integrable random value with respect to this set of measures  and such that 
\begin{eqnarray}\label{r7}
 E^{P_i}\xi=M_0, \quad i=\overline{1,k},
\end{eqnarray}
then the random process  $\{M_m=\sup\limits_{P\in M}E^P\{\xi|{\cal F}_m\}, {\cal F}_m\}_{m=0}^\infty$ is a martingale relative to a convex set of equivalent measures  $M.$
\end{lemma}
\begin{proof} Due to Lemma  \ref{q5} a random process $\{M_m=\sup\limits_{P\in M}E^P\{\xi|{\cal F}_m\}, {\cal F}_m\}_{m=0}^\infty $  is a supermartingale, that is, 
$$ E^P\{M_m | {\cal F}_{m-1}\} \leq M_{m-1}, \quad m=\overline{1, \infty}, \quad P \in M.$$
Or, $E^PM_m \leq M_0.$
From the other side
$$ E^{P_s} [ \max\limits_{1\leq i \leq k}E^{P_i}\{\xi|{\cal F}_m\}] \ge 
\max\limits_{1\leq i \leq k}E^{P_s}E^{P_i}\{\xi|{\cal F}_m\} \geq M_0, \quad s=\overline{1,k}.$$
 The above inequalities imply $E^{P_s}M_m= M_0, \  m=\overline{1, \infty}, \ s=\overline{1, k}.$
The last equalities lead to equalities $E^{P}M_m= M_0, \ m=\overline{1, \infty}, \ P \in M.$
The fact that $M_m$ is a supermartingale relative to the set of measures $M$ and the above equalities  prove Lemma \ref{1q5}.
\end{proof}

\begin{thm}\label{mars12}
 Let a filtration ${\cal F}_n$  and a set of equivalent measures $\{P_1, \ldots, P_k \}$ on a measurable space $\{ \Omega, {\cal F}\}$ satisfy conditions $A$ and $B,$ correspondingly. 
Suppose that $\xi$ is a nonnegative integrable  random value relative to this set of measures. If $\xi$ is  ${\cal F}_N$-measurable one for a certain $N<\infty,$ then
 a supermartingale $\{f_m, {\cal F}_m\}_{m=0}^\infty, $ where 
$$f_m=\sup\limits_{P \in M}E^P\{\xi| {\cal F}_m\}, \quad  m=\overline{1,\infty}, \quad   \max\limits_{1 \leq i \leq k}E^{P_i}\xi< \infty, $$
 is  local regular one  if and only if
\begin{eqnarray}\label{mars13}
E^{P_i}\xi=f_0, \quad i=\overline{1,k}.
\end{eqnarray}
\end{thm}
\begin{proof} { \bf The necessity.} Let $\{f_m, {\cal F}_m\}_{m=0}^\infty $  be a local regular supermartingale. Then there exists a sequence of nonrandom stopping times $\tau_s=n_s, \ s=\overline{1, \infty},$ such that for every $n_s$ there exists
$ \varphi=\sum\limits_{m=1}^{n_s}\sum\limits_{i=1}^{k}E^{P_i}\{\xi|{\cal F}_m\}$ satisfying inequalities
$$\max\limits_{1 \leq j \leq k}E^{P_j}\varphi \leq \sum\limits_{m=1}^{n_s}\sum\limits_{i=1}^{k}\max\limits_{1 \leq j \leq k}E^{P_j}
E^{P_i}\{\xi|{\cal F}_m \}\leq$$
$$\sum\limits_{m=1}^{n_s}\sum\limits_{i=1}^{k}\max\limits_{1 \leq j \leq k}E^{P_j} \max\limits_{1 \leq i \leq k} E^{P_i}\{\xi|{\cal F}_m\} \leq$$
$$ \sum\limits_{m=1}^{n_s}\sum\limits_{i=1}^{k}\max\limits_{1 \leq j \leq k}E^{P_j} \max\limits_{1 \leq i \leq k} E^{P_i} \xi =n_s k  \max\limits_{1 \leq i \leq k} E^{P_i} \xi,$$
$$\sup\limits_{P \in M}E^P\varphi \leq  \max\limits_{1 \leq j \leq k}E^{P_j}\varphi\leq n_s k  \max\limits_{1 \leq i \leq k} E^{P_i} \xi,$$
and nonnegative adapted   random process $\{\bar g_m^0\}_{m=0}^\infty, \ \bar g_0^0=0, \ $
$E^{P_i}\bar g_m^0< \infty, \ 0 \leq m \leq n_s$ such that 
$$ f_m +\sum\limits_{i=1}^m\bar g_i^0=\bar M_m, \quad E^{P}\bar M_m=f_0, \quad 0 \leq m \leq n_s, \quad P \in M.$$ 
If $n_s >N,$ then 
$$E^{P_i}(\xi+\sum\limits_{i=1}^N \bar g_i^0)=E^{P_i}\xi+ E^{P_i}\sum\limits_{i=1}^N\bar g_i^0=f_0.$$
But there exists $1 \leq  i_1 \leq k$ such that $E^{P_{i_1}}\xi=f_0.$ Therefore,
$E^{P_{i_1}}\sum\limits_{i=1}^Ng_i^0=0.$ Due to equivalence of measures $P_i, \ i=\overline{1,k},$ we obtain 
\begin{eqnarray}\label{mars14}
E^{P_i}\xi=f_0, \quad i=\overline{1,k},
\end{eqnarray}
 where $f_0=\sup\limits_{P \in M}E^P\xi.$

{\bf Sufficiency.} If conditions  (\ref{mars14}) are satisfied,  then $\bar M_m= \sup\limits_{P \in M}E^P\{\xi|{\cal F}_m\}$ is a martingale. The last implies local regularity of $\{f_m, {\cal F}_m\}_{m=0}^\infty.$
The Theorem \ref{mars12} is proved.
\end{proof}

\begin{thm}\label{mars5}
Let a filtration ${\cal F}_n$ on a measurable space $\{\Omega, {\cal F}\}$ satisfies condition $A $ and let $M$ be a convex set of equivalent measures on this measurable space.
 Suppose that
\begin{eqnarray*}\label{mn112}
P(A_s^n) >0, \quad P \in M,  \quad  s=\overline{1, \infty}, \quad   n=\overline{1, \infty},
\end{eqnarray*}
 and for a certain measure   $P_{ i_0} \in M $ the inequalities
\begin{eqnarray*}\label{mn112}
\frac{P(A_j^{n+1})}{P(A_s^n)} \leq \frac{P_{i_0}(A_j^{n+1})}{P_{i_0}(A_s^n)}, \quad i=\overline{1,k}, \quad A_j^{n+1} \subseteq A_s^n, \quad  A_s^n=\bigcup\limits_{j \in I_s^n}A_j^{n+1}, \quad  P \in M,
\end{eqnarray*}
are valid.
If  $G_0$ is a set of all integrable nonnegative random values $\xi$ satisfying conditions
\begin{eqnarray}\label{mars6}
E^P\xi=1, \quad P \in M,
\end{eqnarray}
 then the random process $\{ E^P\{\xi|{\cal F}_m\}, {\cal F}_m\}_{m=0}^\infty, \ \xi \in G_0, $  is a local regular supermartingale.
\end{thm}
\begin{proof}  Let $P_1, \ldots, P_n$ 
be a certain subset of measures from $M$ containing the measure $P_{i_0}.$ Denote by
$M_n$  a convex set of equivalent measures 
\begin{eqnarray}\label{mars8}
M_n = \{P \in M, \  P=\sum\limits_{i=1}^n\alpha_i P_i, \  \alpha_i \geq 0, \ i=\overline{1,n}, \ \sum\limits_{i=1}^n\alpha_i=1\}.
\end{eqnarray}
Due to Lemma \ref{1q5} $\{\bar M_m\}_{m=0}^\infty$ is a  martingale relative
to the set of measures $M_n,$ where $ \bar M_m=\sup\limits_{P \in M_n}E^P\{\xi|{\cal F}_m\}, \ \xi \in G_0.$  Let us consider an arbitrary measure $P_0 \in M$ and let
\begin{eqnarray}\label{mars9}
M_n^{P_0} = \{P \in M, \ P= \sum\limits_{i=0}^n\alpha_i P_i, \  \alpha_i \geq 0, \  i=\overline{0,n}, \ \sum\limits_{i=0}^n\alpha_i=1\}.
\end{eqnarray}
Then  $\{\bar M_m^{P_0}\}_{m=0}^\infty,$ where $\bar M_m^{P_0}=\sup\limits_{P \in M_n^{P_0}}E^P\{\xi|{\cal F}_m\},$ is a martingale relative to the set of measures $M_n^{P_0}.$ It is evident that
\begin{eqnarray}\label{mars10}
\bar M_m \leq \bar M_m^{P_0}, \quad   m=\overline{0, \infty}.
\end{eqnarray}
Since $E^P\bar M_m=E^P\bar M_m^{P_0}=1, \ m=\overline{0, \infty}, \ P \in M_n,$ the inequalities (\ref{mars10}) give $\bar M_m=\bar M_m^{P_0}.$ Analogously,
$E^{P_0}\{\xi|{\cal F}_m\} \leq \bar M_m^{P_0}.$ From equalities 
$ E^{P_0}E^{P_0}\{\xi|{\cal F}_m\}$ $ = E^{P_0}\bar M_m^{P_0}=1$ we obtain
$E^{P_0}\{\xi|{\cal F}_m\} = \bar M_m^{P_0}=\bar M_m.$ 
Since the measure $P_0$ is arbitrary it implies that $E^P\{\xi|{\cal F}_m\}, \   m=\overline{0, \infty},$ is a martingale relative to all measures from $M.$
Due to Theorem  \ref{mars1}, it is a local regular supermartingale with random process $\bar g^0_m=0,  m=\overline{0, \infty}. $
The Theorem \ref{mars5} is proved.
\end{proof}

\begin{thm}\label{mmars1}
Let a filtration ${\cal F}_n$ on a measurable space $\{\Omega, {\cal F}\}$ satisfies condition $A $ and let $M$ be a convex set of equivalent measures on this measurable space.
 Suppose that 
\begin{eqnarray*}\label{mn111}
 P(A_s^n) >0, \quad P \in M,   \quad  s=\overline{1, \infty},  \quad n=\overline{1, \infty},
\end{eqnarray*}
and for a certain measure   $P_{ i_0} \in M $ the inequalities
\begin{eqnarray*}\label{mn111}
\frac{P(A_j^{n+1})}{P(A_s^n)} \leq \frac{P_{i_0}(A_j^{n+1})}{P_{i_0}(A_s^n)}, \quad i=\overline{1,k}, \quad A_j^{n+1} \subseteq A_s^n, \quad A_s^n=\bigcup\limits_{j \in I_s^n}A_j^{n+1}, \quad P \in M,
\end{eqnarray*}
are valid.
If $\{f_m, {\cal F}_m\}_{m=0}^\infty$ is an adapted random process  satisfying conditions
\begin{eqnarray}\label{mmars2}
f_m \leq f_{m-1}, \quad  E^P\xi|f_m| <\infty, \quad P \in M \quad m=\overline{1, \infty}, \quad  \xi \in G_0,
\end{eqnarray}
where $G_0=\{\xi \geq 0, E^P\xi=1, P \in M\},$ 
 then the random process
\begin{eqnarray}\label{mmars3}
 \{ f_mE^P\{\xi|{\cal F}_m\}, {\cal F}_m\}_{m=0}^\infty, \quad  P \in M,
\end{eqnarray}
is a local regular supermartingale relative to all measures from $M.$
\end{thm}

\begin{proof} Due to Theorem  \ref{mars5}, the random process 
$\{ E^P\{\xi|{\cal F}_m\}, {\cal F}_m\}_{m=0}^\infty$ is a martingale relative to all measures from $M.$ Therefore,
\begin{eqnarray*}
f_{m-1}E^P\{\xi|{\cal F}_{m-1}\} - E^P\{ f_m E^P\{\xi|{\cal F}_m\}|{\cal F}_{m-1}\}=
\end{eqnarray*}
\begin{eqnarray}\label{mmars4}
E^P\{ (f_{m-1} - f_m) E^P\{\xi|{\cal F}_m\}|{\cal F}_{m-1}\}, \quad m=\overline{1, \infty}.
\end{eqnarray}
So, if to put  $\bar g_m^0= (f_{m-1} - f_m) E^P\{\xi|{\cal F}_m\}, \ m=\overline{1, \infty}, $ then $\bar g_m^0 \geq 0$ and  it is  ${\cal F}_m$-measurable and
$E^P\bar g_m^0 \leq E^P\xi(|f_{m-1}|+|f_m|)< \infty.$ It proves the needed statement.
\end{proof}

\begin{cor} If $f_m=\alpha, \ m=\overline{1, \infty}, \ \alpha \in R^1,$ then 
$\{\alpha E^P\{\xi |{\cal F}_m\}\}_{m=0}^\infty$ is a local regular supermartingale. If $\xi =1,$ then $\{f_m\}_{m=0}^\infty$ is also a local regular supermartingale.
\end{cor}
Denote by $F_0$ the set of adapted processes
\begin{eqnarray*}\label{mmars5}
F_0=\{ f=\{f_m\}_{m=0}^\infty,  \   P(|f_m| <\infty) =1, \ P \in M, \ f_m \leq f_{m-1}, \  m=\overline{1, \infty}\}.
\end{eqnarray*}
For every $\xi \in G_0$ let us introduce the set of adapted processes
$$ L_{\xi}=$$
\begin{eqnarray*}\label{mmars6}
\{\bar f=\{f_mE^P\{\xi|{\cal F}_m\}\}_{m=0}^\infty, \  \{f_m\}_{m=0}^\infty \in F_0, \   E^P\xi|f_m| <\infty, \ P \in M, \ m=\overline{1, \infty}\},
\end{eqnarray*}
and 
\begin{eqnarray*}\label{mmars7}
V=\bigcup\limits_{\xi \in G_0}L_{\xi}.
\end{eqnarray*}

\begin{cor} Every random adapted process from the set $K,$ where
\begin{eqnarray*}\label{mmars8}
K=\left \{ \sum\limits_{i=1}^mC_i \bar f_i, \ \bar f_i \in V, \  C_i \geq 0, \ i=\overline{1, m}, \ m=\overline{1, \infty}\right\}, 
\end{eqnarray*}
 is a local regular supermartingale.
\end{cor}
\begin{proof} The proof is evident.
\end{proof}

\begin{thm}\label{mmars9}
Let a filtration ${\cal F}_n$ on a measurable space $\{\Omega, {\cal F}\}$ satisfies condition $A $ and let $M$ be a convex set of equivalent measures on this measurable space.
Suppose that 
\begin{eqnarray*}\label{mn115}
  P(A_s^n) >0, \quad P \in M,   \quad  s=\overline{1, \infty}, \quad  n=\overline{1, \infty},
\end{eqnarray*}
and  for a certain measure   $P_{ i_0} \in M $ the inequalities 
 \begin{eqnarray*}\label{mn115}
\frac{P(A_j^{n+1})}{P(A_s^n)} \leq \frac{P_{i_0}(A_j^{n+1})}{P_{i_0}(A_s^n)}, \quad i=\overline{1,k}, \quad A_j^{n+1} \subseteq A_s^n, \quad A_s^n=\bigcup\limits_{j \in I_s^n}A_j^{n+1}, \quad P \in M,
\end{eqnarray*}
are valid.
If $\{f_m\}_{m=0}^\infty$ is a nonnegative uniformly integrable supermartingale relative to the set of measures from $M$, then 
the necessary and sufficient conditions for it  to be a local regular one is belonging it to the set $K.$
\end{thm}
\begin{proof}
{\bf Necessity.}  It is evident  if  $\{f_m\}_{m=0}^\infty$ belongs to $K$ then it is a local regular supermartingale.

{\bf Sufficiency.} Suppose that  $\{f_m\}_{m=0}^\infty$ is a local regular supermartingale. Then there exists  nonnegative adapted process $\{\bar g_m^0\}_{m=1}^ \infty,  \ E^P\bar g_m^0< \infty, \ m=\overline{1, \infty}, $  and a martingale  $\{M_m\}_{m=0}^ \infty,$
such that 
\begin{eqnarray*}\label{mmars8}
f_m=M_m - \sum\limits_{i=1}^m\bar g_i^0, \quad  m=\overline{0, \infty}. 
\end{eqnarray*}
Then $M_m \geq 0, \ m=\overline{0, \infty}, \ E^P M_m <\infty, \ P\in M.$
Since $0< E^PM_m=f_0< \infty$ we have $E^P\sum\limits_{i=1}^m\bar g_i^0< f_0.$ Let us put $g_{\infty}=\lim\limits_{m \to \infty}\sum\limits_{i=1}^m\bar g_i^0.$
Using uniform integrability of $f_m$ we can pass to the limit in the equality
$$ E^P(f_m +\sum\limits_{i=1}^m\bar g_i^0)=f_0, \quad P \in M,$$
as $m \to \infty$.
Passing to the limit in the last equality, as $m \to \infty,$  we obtain
$$E^P(f_\infty +g_{\infty})=f_0.$$
Introduce into consideration a  random value $\xi=\frac{f_\infty +g_{\infty}}{f_0}.$
Then $E^P\xi=1, \ P \in M.$   From here we obtain that  $\xi \in G_0$ and 
$$M_m=f_0E^P\{\xi|{\cal F}_m\}, \ m=\overline{0, \infty}.$$
 Let us put $\bar f_m^2=-\sum\limits_{i=1}^m\bar g_i^0. $ It is easy to see that an adapted random process $\bar f_2=\{\bar f_m^2\}_{m=0}^\infty$ belongs to $F_0.$ Therefore,
for the supermartingale $f=\{ f_m\}_{m=0}^\infty$ the representation 
$$f=\bar f_1+ \bar f_2,$$
is valid, where $\bar f_1=\{f_0E^P\{\xi|{\cal F}_m\}\}_{m=0}^ \infty$  belongs to $L_{\xi}$
with  $ \xi = \frac{f_\infty +g_{\infty}}{f_0}$  and $ f_m^1=f_0, \ m=\overline{0,\infty}.$ The same is valid for $\bar f_2$ with $\xi=1.$ This implies that $f$ belongs to $K.$ The Theorem 
\ref{mmars9} is proved.
\end{proof}

Below we present some results nedeed for the description of the set $G_0.$ We consider the case, as conditions of the Theorems \ref{mars5}, \ref{mmars1}, \ref{mmars9} are valid. 
Let us consider the set of equations for a certain fixed $n\geq 1$
 \begin{eqnarray}\label{map1}
 \sum\limits_{j=1}^{\infty} P_i(A_j^n)\xi_j=1, \quad i =\overline{1,k}.
\end{eqnarray} 
If there exists nonnegative solution $\{\xi_j\}_{j=1}^\infty$ of the set of equations
(\ref{map1}), then the  random value
$\xi=\sum \limits_{j=1}^\infty\xi_j\chi_{A_j^n}$ is ${\cal F}_n$-measurable and belongs to the set $G_0.$

If to put $a_j=\{P_i(A_j^n)\}_{i=1}^k, \ j=\overline{1, \infty},$ then the set of equations (\ref{map1})   can be written in the form 
 \begin{eqnarray}\label{mapp1}
 \sum\limits_{j=1}^{\infty} a_j \xi_j=a_0
\end{eqnarray}
 with the vector $a_0=\{e_i\}_{i=1}^k, \  e_i=1, \ i=\overline{1,k}.$
It is evident that homogeneous set of equations
 \begin{eqnarray}\label{maapp1}
 \sum\limits_{j=1}^{\infty} a_j \xi_j=0
\end{eqnarray}
has always a bounded nonzero solution. Then if to denote it by $u=\{u_j\}_{j=1}^\infty,$ then due to  boundedness of this solution, that is, $|u_j|\leq C < \infty, \ j=\overline{1, \infty},$ there exists a real number $t>0$ such that $ \xi_j=1- t u_j \geq 0, \ j=\overline{1, \infty}.$  Such a vector $\{ \xi_j\}_{j=1}^\infty$ is a nonhomogeneous nonnegative solution to the set of equations (\ref{mapp1}).

Bellow we prove  Theorem \ref{ttt9} helping  us to describe strictly positive solutions of the set equations (\ref{mapp1}).
\begin{defin} A vector $a_0 \in R_+^k$ belongs to the  interior of the nonnegative cone generated by vectors $a_j \in R_+^k, \ j=\overline{1, \infty},$ if there exist positive numbers $\alpha_j>0, j=\overline{1, \infty},$ such that  
\begin{eqnarray}\label{ap1}
 \sum\limits_{j=1}^{\infty}\alpha_j a_j=a_0.
\end{eqnarray} 
\end{defin}

The next Theorem generalizes a Theorem from  \cite{Gonchar2} and describes all strictly positive solutions to the set of equations (\ref{mapp1}). 
\begin{thm}\label{ttt9} 
Let a vector  $a_0$  belongs to the interior of the cone generated by vectors $a_j \in R^k, \ j=\overline{1, \infty},$ 
were dimension of the  cone is    $ 1 \leq r \leq k,$ and let $r$ linear independent  vectors   $a_1, \ldots,a_{r}$ be  such that  the vector
 $a_0$  belongs to the  interior of the cone generated by these vectors. Then there  exists infinite number of  linear independent  nonnegative solutions    $z_i, \ \ i=\overline{r, \infty},$ of the set of equations (\ref{mapp1}), where
 $$ z_{r}=\{\left\langle a_0, f_1 \right\rangle, \ldots, \left\langle a_0, f_r \right\rangle, 0, 0, \ldots,\},   $$
$$ z_{i}=\{\left\langle a_0, f_1 \right\rangle - \left\langle a_i, f_1 \right\rangle y_i^*, \ldots, \left\langle a_0, f_r \right\rangle - \left\langle a_i, f_r \right\rangle y_i^*, 0, \ldots, 0,y_i^*, 0,\ldots,  \}, $$ $$  i=\overline{r+1, \infty},   $$ 
$$ y_i^*=\left\{\begin{array}{l l} \min\limits_{l \in K_i}\frac{\langle a_0, f_l \rangle }{\langle a_i, f_l \rangle }, & K_i=\{l, \langle a_i, f_l\rangle >0\}, \\ 
 1, & \langle a_i, f_l\rangle  \leq 0,\  \forall l=\overline{1,r},  
                                                 \end{array} \right. $$
$ \{f_1, \ldots, f_k\}$ is a set of linear independent vectors satisfying conditions 
\begin{eqnarray}\label{ap2}                                                 
\langle f_i, a_j\rangle=\delta_{ij}, \quad i,j=\overline{1, r}, \quad \langle f_i, a_j\rangle=0, \quad
j=\overline{1, r}, \ i=\overline{r+1, k}.
\end{eqnarray}
The set of strictly positive solutions of the set of equations (\ref{mapp1}) is given by the formula
\begin{eqnarray}\label{ap3}
z= \sum\limits_{i=r}^{\infty}\gamma_iz_i,
\end{eqnarray}
where the vector  $\gamma=\{\gamma_r, 
\ldots, \gamma_i, \ldots,\}$ satisfies conditions 
$$\sum\limits_{i=r}^{\infty}\gamma_i=1, \quad \gamma_i>0, \quad  i=\overline{r+1, \infty}, \quad \sum\limits_{i=r+1}^{\infty}a_i\gamma_iy_i^* < \infty, $$ 
\begin{eqnarray}\label{ap4}
 \left\langle a_0 - \sum\limits_{i=r+1}^{\infty}a_i\gamma_iy_i^*, f_k \right\rangle>0, \quad k=\overline{1, r}.
\end{eqnarray}
\end{thm} 
 \begin{proof} In the Theorem \ref{ttt9}  without loss of generality we assume that $r$ linear independent vectors $a_1, \ldots,a_{r}$ are such that the vector $a_0$ belongs to the interior of the cone generated by these vectors.  If it is not the case and such vectors are  $a_{i_1}, \ldots,a_{i_r},$ then by the renumbering the set of the vectors $a_j, \ j=\overline{1, \infty}, $ we come to the case of the Theorem  \ref{ttt9}.

Let us indicate the necessary conditions of the  existence of  strictly positive solution to the set of equations  (\ref{mapp1}). Due to existence of nonnegative  solution of (\ref{mapp1}),  the series 
 $ \sum\limits_{i=1}^{\infty} \xi_i a_i$ is  convergent one. Since 
  $  a_i \in R_+^d$   we have  that the series  $ \sum\limits_{i=r+1}^{\infty}\xi_i a_i$ is also  convergent one. Denote by $\{f_1, \ldots, f_d\}$  a set of vectors that satisfy conditions (\ref{ap2}). We obtain  that a set of equations  (\ref{mapp1})  is equivalent to the set of equations 
\begin{eqnarray}\label{ap5}
\left\langle \sum\limits_{i=r+1}^{\infty} \xi_i a_i, f_j \right\rangle + \xi_j=\left\langle a_0, f_j   \right\rangle >0, \quad     j=\overline{1,r}. 
\end{eqnarray}
where  $\langle a, b\rangle$  denotes a scalar product of vectors  $a$ and $b.$     
From here we have
\begin{eqnarray}\label{ap6}
\left\langle  a_0 - \sum\limits_{i=r+1}^{\infty}\xi_i  a_i, f_j \right\rangle = \xi_j,  \quad     j=\overline{1,r}. 
\end{eqnarray}
It implies that inequalities   
\begin{eqnarray}\label{ap7}
\left\langle  a_0 - \sum\limits_{i=r+1}^{\infty}\xi_i a_i, f_j \right\rangle >0,  \quad     j=\overline{1,r}, 
\end{eqnarray}
are valid. 
If strictly positive vector $\{\xi_{r+1}, \ldots, \xi_m, \ldots \}$ is such that the  series 
 $ \sum\limits_{i=1}^{\infty}\xi_i a_i$ is convergent one and inequalities  (\ref{ap7}) are valid, then the vector 
 \begin{eqnarray*}\label{ap8} 
 z=\left \{\left\langle  a_0 - \sum\limits_{i=r+1}^{\infty}\xi_ia_i, f_1 \right\rangle, \ldots, \left\langle  a_0 - \sum\limits_{i=r+1}^{\infty} \xi_ia_i, f_r \right\rangle, \xi_{r+1}, \ldots, \xi_l, \ldots, \right\}
\end{eqnarray*}
is a general strictly positive solution of   the set of equations  (\ref{mapp1}). It is evident that nonnegative solution  $z_l$  we obtain from the general strictly positive solution of   (\ref{mapp1}), if to put   $\{\xi_{r+1}, \ldots, \xi_m, \ldots \}$  such that   $\xi_i=0, \ i\neq l, \xi_l=y_l^*.$ 
These solutions are nonnegative and linear independent. It is evident if to choose the vector  
$\gamma=\{\gamma_r, \ldots, \gamma_i, \ldots,\} $ such that 
 \begin{eqnarray}\label{ap9}
 \sum\limits_{i=r+1}^{\infty}a_iy_i^* \gamma_i< \infty, \quad \sum\limits_{i=r}^{\infty}\gamma_i=1, \quad \gamma_i > 0, \quad i=\overline{r, \infty},
\end{eqnarray} 
then we obtain that inequalities  
 \begin{eqnarray}\label{ap10}
 \left\langle  a_0 - \sum\limits_{i=r+1}^{\infty}a_iy_i^* \gamma_i, f_j \right\rangle>0, \quad j=\overline{1,r}.
\end{eqnarray} 
are valid.
From here a vector   $ \sum\limits_{i=r}^{\infty}\gamma_iz_i$ is strictly positive solution of the set of equations  (\ref{mapp1}). 

It is evident that these conditions are also sufficient.
Theorem \ref{ttt9} is proved.
\end{proof}

It is easy to see that the vector $a_0$ belongs to the  interior of the cone generated by vectors $a_j=\{P_i(A_j^n)\}_{i=1}^k, \ j=\overline{1, \infty}.$
 The existence  
of $r$ linear independent subset of vectors $\{a_{i_1},\ldots, a_{i_r} \}$ from the set of vectors  $a_j=\{P_i(A_j^n)\}_{i=1}^k, \ j=\overline{1, \infty},$ such that the vector $a_0$ belongs to the interior of the cone generated by this subset of vectors is the conditions on the set of  measures $\{P_1, \ldots, P_k\}.$ A simple criterion of verifying of belonging to the interior of the cone a certain vector $a_0$ is contained in \cite{Gonchar2}.

At last, let us give an example of measurable space $\{\Omega, {\cal F}\}$ and filtration  on it and also a set of measures $P_1, \ldots, P_k$ satisfying conditions $A$ and $B.$ Let us put $\Omega=[0,1).$  Choose any monotonously increasing sequence  $\{x_k\}_{k=0}^\infty,$ such that $x_0=0, x_k <x_{k+1}, \lim\limits_{k\to \infty}x_k=1.$ Denote by $A_s^1=[a_s^1,b_s^1)=[x_{s-1},x_s), s=\overline{1, \infty}.$ The sets  $A_s^2, s=\overline{1, \infty},$ we construct by dividing in half  intervals $A_s^1$ and so on. Let us give  measures $P_1, \ldots, P_k$ on ${\cal F}_n$ generated by sets $A_s^n, s=\overline{1, \infty}.$ On Borel $\sigma$-algebra ${\cal B}([0,1))$ of the set $[0,1)$ let us give a set of 
measures  $P_1, \ldots, P_k$  by their Radon-Nicodym derivatives $\frac{dP_i}{dP_1}=i x^{i-1}, x \in [0,1), i=\overline{1,k}, $ where $P_1$ is Lebesgue measure on $[0,1).$  Consider restrictions of this measures on the $\sigma$-algebra ${\cal F}_n.$ 
It easy to see that so given measures on ${\cal F}_n$ satisfy condition $B$ with index $i_0=1.$

Applications of the results obtained to Mathematical Finance will be given in separated paper. 

\bibliographystyle{vmsta-mathphys}
\bibliography{biblio}

\end{document}